\documentclass[10pt, reqno]{amsart}              

\usepackage{amsmath}
\usepackage{amscd, amsthm, amsfonts, amssymb, xfrac}
\usepackage{mathrsfs, bbm, cancel, stmaryrd}
\usepackage{marvosym, enumerate, enumitem, caption, varwidth, slantsc, lmodern}
\usepackage[all]{xy}
\usepackage{graphicx}
\usepackage{rotating, lscape, longtable, multirow, array, afterpage, color}

\usepackage{hyperref}
\hypersetup{}

\newtheorem{prop}{Proposition}[section]
\newtheorem{thm}[prop]{Theorem}
\newtheorem{lem}[prop]{Lemma}

\theoremstyle{remark}

\theoremstyle{definition}
\newtheorem{defn}[prop]{Definition}
\newtheorem{nota}[prop]{Notation}
\newtheorem{coro}[prop]{Corollary}
\newtheorem{rem}[prop]{Remark}

\newcommand{\cl}{{\rm C\ell}}

\newcommand{\q}{{\rm q}}

\newcommand{\ep}{\epsilon}

\newcommand{\ZZ}{\mathbbm Z}
\newcommand{\RR}{\mathbbm R}

\newcommand{\CC}{\mathbbm C}

\newcommand{\ei}{\mathbbm{1}}

\allowdisplaybreaks

\renewcommand{\arraystretch}{1.3}

\begin{document}

\title{Connections on Cahen-Wallach spaces}
\author{Frank Klinker}
\date{\today}

\thanks{\vspace*{0.5ex}{\em Adv. Appl. Clifford Algebr.} (2014), 32 pp. 
\href{http://dx.doi.org/10.1007/s00006-014-0451-7}{http://dx.doi.org/10.1007/s00006-014-0451-7}} 
\address{Faculty of Mathematics, TU Dortmund University, 44221 Dortmund, Germany} 
\email{ \href{mailto:frank.klinker@math.tu-dortmund.de}{frank.klinker@math.tu-dortmund.de}}
\begin{abstract} 
We systematically discuss connections on the spinor bundle of Cahen-Wallach symmetric spaces. A large class of these connections is closely connected to a quadratic relation on Clifford algebras. This relation in turn is associated to the symmetric linear map that defines the underlying space. We present various solutions of this relation. Moreover, we show that the solutions we present here provide a complete list with respect to a particular algebraic condition on the parameters that enter into the construction.
\end{abstract}
\maketitle
\setcounter{tocdepth}{1}
\tableofcontents


\section{Introduction}

In this text we examine spinor connections on solvable Lorentzian symmetric spaces and a closely related quadratic relation on the Clifford algebra of a Euclidean vector space, $V$. One parameter that connects the two topics is a symmetric endomorphism $B\in{\rm End}(V)$. 
Solvable Lorentzian symmetric spaces have been classified in \cite{CW70} and thus we will use the term Cahen-Wallach-spaces or CW-space. From a differential geometric point of view CW-spaces are considered in the classification of holonomy groups of Lorentzian manifolds, see \cite{Baum12, Leist07}. In physics they are used as backgrounds of supergravity theories because they have a large space of isometries and may admit many supersymmetry generators, see for example \cite{CheKo84, Fig1, MeFig04, FigPapado1, Hustler, Meessen2002}. In the supersymmetry setting CW-spaces are sometimes called KG-spaces or just $pp$-waves, although they form a special class among them. 
In all applications connections on CW-spaces play a major role, but a systematic treatment is missing. This is the aim of the text at hand.

In the first part of our text we present the above mentioned quadratic relation on a Clifford algebra, see Definition \ref{cp}. We develop solutions of this relation: the quadratic Clifford pairs. We do this by investigating certain special approaches. As a result we get solutions for any symmetric map $B$ according to some restriction on its eigenvalue structure, see Propositions \ref{monomial}, \ref{prop2}, \ref{hom}, and \ref{hom-p}.

The second part is on the origin of the quadratic relation. As we will see, it naturally arises when we consider connections on the spinor bundle of CW-spaces. 
We will recall some basic facts on invariant connections on homogeneous spaces and especially on CW-spaces. Therefore, we introduce a notation that is well adapted to the algebraic nature of the topic, see Definition \ref{Cliffordmap}. The quadratic relation then is shown to correspond to the vanishing of the curvature of a particular class of connections. This class is defined by a trivial action of the center of the CW-space, see (\ref{20}). For the result see Theorem \ref{alphazero} and we remark that this has been found to be true before by the author in \cite{santi1}. We will also give the result for a general connection, see Theorem \ref{alphanotzero}. Furthermore, we will discuss certain restrictions of the space of spinors on which the Clifford-algebra acts and we show that some of the properties of the connections survive.

In a last part we will discuss an algebraic condition that turns the examples of quadratic Clifford pairs that we found in the first part into a complete list. In this respect we classify flat homogeneous connections on irreducible spinor bundles over CW-spaces, see Proposition \ref{c+d} and Theorem \ref{thm-c+d}.

\section{Quadratic Clifford pairs}
\label{sec:qcp}

Let us consider an finite dimensional real  vector space $V$ with Eu\-clid\-ean metric $g$. We denote by $\cl(V)$ its complex Clifford algebra\footnote{For more details on Clifford algebras, aka geometric algebras that go beyond the facts we state in this text we recommend the books \cite{Chevalley},  \cite{Harvey}, \cite{LawMich}, or \cite{Lou}} subject to the Clifford relation 
\begin{equation}\label{cliff}
vw+wv=-2g(v,w)
\end{equation} 
for all $v,w\in V$.
\begin{defn}\label{cp}
Let $(V,g)$ be a Euclidean vector space and $\cl(V)$ as above. For $a,b\in \cl(V)$ consider the map 
$q_{a,b}:\cl(V)\to \cl(V)$ that is defined by 
\begin{equation}\label{quad1}
q_{a,b}(x):=a^2x+ xb^2 -2axb\,.
\end{equation}
A {\em quadratic Clifford pair} is a pair $(a,b)\in \cl(V)\times \cl(V)$ that obeys
\begin{enumerate}
\item[(i)] $q_{a,b}V\subset V$ and 
\item[(ii)] $q_{a,b}$ restricted to $V$ is a $g$-symmetric element in ${\rm End}(V)$.
\end{enumerate}
We call the quadratic Clifford pair {\em (non)-degenerate} if $q_{a,b}\big|_V$ is a (non)-degene\-rate symmetric map and we call it {\em associated to $B\in {\rm End}(V)$} if $q_{a,b}\big|_V=B$.
\end{defn}
For $a,b\in \cl(V)$ the map $q_{a,b}$ is the square of the map $s_{a,b}$ with $s_{a,b}(x):=ax-xb$. 
Therefore, we have 
\begin{equation}\label{sum}
q_{a_1+a_2,b_1+b_2}=q_{a_1,b_1}+q_{a_2,b_2}+s_{a_1,b_1}\circ s_{a_2,b_2}+s_{a_2,b_2}\circ s_{a_1,b_1}
\end{equation} for all $a_1,a_2,b_1,b_2\in \cl(V)$ .  
 
\begin{rem}\label{a=0}
We write $a,b\in \cl(V)$ as $a=\alpha+a'$ and $b=\beta+b'$ with $\alpha$ and $\beta$ being the scalar parts of $a$ and $b$, respectively. Then $q_{a,b}$ reads as
\begin{equation}
q_{\alpha+a',\beta+b'}(x)=(\alpha-\beta)^2x+q_{a',b'}(x)-2(\alpha-\beta)s_{a',b'}(x)\,.
\end{equation}
This means that $q_{a,b}$ is invariant under the action of the scalars given by $(a,b)\mapsto (a+\alpha,b+\alpha)$. Therefore, we may assume that the scalar part of $a$ or $b$ vanishes when we look for quadratic Clifford pairs $(a,b)$. 
\end{rem}

Let $\Gamma$ be an irreducible matrix representation of $\cl(V)$ and let $\{\Gamma_\mu\}$ be a set of $\gamma$-matrices associated to a basis $\{e_\mu\}$ of $V$, i.e.\ $\Gamma_\mu=\Gamma(e_\mu)$ with $\Gamma_\mu\Gamma_\nu+\Gamma_\nu\Gamma_\mu=-2g_{\mu\nu}\mathbbm{1}$. 
When $q_{a,b}$ obeys (ii) this basis can be chosen as an orthonormal basis of eigenvectors of $q_{a,b}$. 

We use the following abbreviation for products of $\gamma$-matrices. Consider the multi-index $I=(i_1,\ldots,i_k)\in\{1,\ldots,n\}^k$ for $k=0,\ldots,n$, then 
\[
\Gamma_I= \Gamma_{\mu_1\ldots \mu_k}:=\Gamma_{[\mu_1}\Gamma_{\mu_2}\cdots\Gamma_{\mu_k]}\,.
\]
For $k=0$ we set $\Gamma_{\emptyset}=\mathbbm{1}$ and we write $\hat I$ for the length of the index, i.e.\ $\hat I:=k$  for $I\in\{1,\ldots,n\}^k$. If $\Gamma_\mu$ is associated to an orthonormal basis and $I$ and $J$ contain the same indices we have $\Gamma_I=\epsilon_J^I\Gamma_J$, where $\epsilon_J^I$ is the sign of  permutation from $I$ to $J$. If we do not care about the ordering of the indices within $I$ we identify the multi-index with the set of indices itself. Therefore, the notations $I\cap J$, $I\cup J$ and $I^\complement$ make sense. Moreover, due to (\ref{cliff}) we have $\Gamma_I^2=\sigma_I\mathbbm{1}$ with $\sigma_I\in\{\pm1\}$. For the combination of $\gamma$-matrices of maximal length we write
\[
\Gamma^*=i^{\left[\frac{n+1}{2}\right]}\Gamma_{1\ldots n}.
\]
Then $\Gamma^*$ is the image of the complex volume form and obeys $(\Gamma^*)^2=\ei$. 

\begin{rem}\label{aut}
We shortly recall two important involutions of the Clifford algebra: By $\bar{\ } :\cl(V)\to \cl(V)$ we denote the extension of $V\ni v\to -v\in V$ to an automorphism of $\cl(V)$ and by $\tilde{\ }:\cl(V)\to \cl(V)$ we denote the extension of $V\ni v\mapsto  v\in V$ to an antiautomorphism of $\cl(V)$. For $I=(\mu_1,\ldots ,\mu_{\hat I})$ we get $\bar\Gamma_I=(-)^{\hat I}\Gamma_I$ and $\tilde\Gamma_I=\Gamma_{\mu_{\hat I}\ldots\mu_1}=(-1)^{\frac{\hat I(\hat I-1)}{2}}\Gamma_I$. The elements in the $+1$ and $-1$ eigenspaces of $\bar{\ }$ are called {\em even} and {\em odd}, respectively.
\end{rem}
Due to relation (\ref{cliff}) a generic element $c\in \cl(V)$ can be written as\footnote{When we consider such expansions the sum is taken over indices of the form $I=(i_1,\ldots,i_{\hat I})$ with $i_1<i_2<\cdots<i_{\hat I}$.}
\begin{equation*}
c=\sum_{I} c_I\Gamma_I, c_I\in \CC\,.
\end{equation*}

\begin{rem}
Suppose $c$ and $d$ are both even or odd with respect to the $\mathbbm{Z}_2$-grading of $\cl(V)$ and $c,d$ is a quadratic Clifford pair. Then so is $(-\Gamma^*c,\Gamma^*d)$ with $q_{c,d}=(-1)^{{\rm deg}_\ZZ c}q_{-\Gamma^*c,\Gamma^*d}$ if $\dim(V)$ is even, and $(\Gamma^*c,\Gamma^*d)$ with $q_{c,d}=q_{\Gamma^*c,\Gamma^*d}$ if $\dim(V)$ is odd.
\end{rem}

\subsection{Traces and symmetry}\label{traces}

Let us identify $\cl(V)={\rm End}(S)$ via $\Gamma$ where $S$ is an irreducible Clifford module for $\cl(V)$. Then the $I$-component of $a=\sum_Ia_I\Gamma_I\in \cl(V)$ can be evaluated by the trace via
\begin{equation}
a_I = (-1)^{\frac{\hat I(\hat I+1)}{2}}\frac{1}{{\rm tr}(\mathbbm{1})} {\rm tr}\left(a\Gamma_I\right)\,.
\end{equation}
The trace yields an inner product on $\cl(V)$ by $\langle a,b\rangle={\rm tr}(a\tilde b)$ where for $b=\sum_Ib_I\Gamma_I$ is given by $\tilde b=\sum_I(-1)^{\frac{\hat I(\hat I-1)}{2}}b_I\Gamma_I$, see Remark \ref{aut}. This product is naturally induced by the inner product $g$ on $V$ because it coincides with the extension of $g$ to the Grassmann algebra $\Lambda^*V^*$. In particular, the basis $\{\Gamma_I\}_{\hat I\geq0}$ obtained from an orthonormal basis $\{e_\mu\}$ of $V$ obeys ${\rm tr}(\Gamma_I\tilde\Gamma_J)=\delta^I_J{\rm tr}(\mathbbm{1})$ with $\delta_I^J=\delta_{\hat I}^{\hat J}\delta_{i_1\ldots i_{\hat I}}^{j_1\ldots j_{\hat J}}$. 

The transpose of $q_{c,d}$ with respect to the induced inner product on $\cl(V)$ is determined by $q_{d,c}$. More precisely,
\begin{align*}
(q_{c,d})_{I}^{J} &= (-)^{\frac{\hat J(\hat J-1)}{2}}{\rm tr}(q_{c,d}(\Gamma_I)\Gamma_J) \\
&	= (-)^{\frac{\hat J(\hat J-1)}{2}}{\rm tr}(c^2\Gamma_I\Gamma_J+\Gamma_I d^2\Gamma_J -2c\Gamma_I d\Gamma_J) \\
&	= (-)^{\frac{\hat J(\hat J-1)}{2}}{\rm tr}(d^2\Gamma_J\Gamma_I+\Gamma_J c^2\Gamma_I -2d\Gamma_J c\Gamma_I) \\
&	= (-)^{\frac{\hat J(\hat J-1)}{2}}{\rm tr}(q_{d,c}(\Gamma_J)\Gamma_I) \\
&	= (-)^{\frac{(\hat I+\hat J)(\hat I+\hat J-1)}{2}}(-)^{\hat I\hat J}(q_{d,c})_{J}^{I} \,.
\end{align*}
In particular, if $(c,d)$ is a quadratic Clifford pair, i.e.\ $q_{c,d}\big|_V$ is a symmetric on $V$, then ${\rm proj}_Vq_{d,c}(v) =q_{c,d}(v)$ on $V$. Moreover, if $(c,d)$ is a quadratic Clifford pair then  $(d,c)$ is a quadratic Clifford pair if and only if  $q_{c,d}$ respects the vector space splitting $\cl_1(V)\oplus\bigoplus_{r\neq1}\cl_r(V)$. All quadratic Clifford pairs that we will introduce soon are of this type.

\subsection{Monomial solutions}\label{sec:quadraticpairs}

Let us consider $w\in V$ such that $q_{w,-w}(v)=B(v)$ with $B=ww^t$. This map is symmetric but degenerate, of course. However, this easy example motivates the first important class of quadratic Clifford pairs. For this, consider homogeneous elements $c=\alpha\Gamma_I$ and $d=\beta\Gamma_J$. Then $c^2=\alpha^2\sigma_{I}\mathbbm{1}$ and $d^2=\beta^2\sigma_{J}\mathbbm{1}$  and  (\ref{quad1}) reduces to
\begin{equation}
q_{c,d}(e_\mu)=(\sigma_{I}\alpha^2+\sigma_{J}\beta^2)\Gamma_\mu
		 -2\alpha\beta\Gamma_I\Gamma_\mu\Gamma_J\,.  \label{quad2}
\end{equation}
We will take a look at the second term in (\ref{quad2}) and consider three cases. We set  $I\cap J=K$, $I=I_0\cup K$, $J=J_0\cup K$, and discuss (a) $\mu\not\in I\cup J$, (b) $\mu\in I_0\cup J_0$, and (c) $\mu\in K$.
We get the following results:

Case (a): We have $\Gamma_I=\epsilon^{I}_{I_0K}\Gamma_{I_0}\Gamma_K$ and $\Gamma_J=\epsilon^{J}_{KJ_0}\Gamma_K\Gamma_{J_0}$ such that 
\begin{equation}
\begin{aligned}\label{a}
\Gamma_I\Gamma_\mu\Gamma_J
	& =\epsilon^{I}_{I_0K}\epsilon^{J}_{KJ_0}\Gamma_{I_0}\Gamma_K\Gamma_\mu\Gamma_K\Gamma_{J_0}
	    =-\sigma_{K}(-1)^{\hat K}\epsilon^{I}_{I_0K}\epsilon^{J}_{KJ_0}\Gamma_{I_0}\Gamma_{\mu}\Gamma_{J_0}\\
	& =(-1)^{\hat I_0}(-1)^{\hat K}\sigma_{K}\epsilon^{I}_{I_0K}\epsilon^{J}_{KJ_0}\Gamma_{\mu I_0J_0}\,.
\end{aligned}
\end{equation}
Case (b): This case is independent of whether  $\mu\in I_0$ or $\mu\in J_0$ and we get for $\mu\in J_0$
\begin{equation}
\begin{aligned}\label{b}
\Gamma_I\Gamma_\mu\Gamma_J
	& =\epsilon^{I}_{I_0K}\epsilon^{J}_{KJ_0}\Gamma_{I_0}\Gamma_K\Gamma_\mu\Gamma_K\Gamma_{J_0}
	    =-\sigma_{ K}\epsilon^{I}_{I_0K}\epsilon^{J}_{KJ_0}\Gamma_{I_0}\Gamma_{\mu}\Gamma_{J_0}\\
	& =-(-1)^{\hat I_0}\sigma_{ K}\epsilon^{I}_{I_0K}\epsilon^{J}_{KJ_0} (\hat I_0+\hat J_0)g_{\mu[i_1}\Gamma_{i_2\ldots i_{\hat I_0}J_0]}\,.
\end{aligned}
\end{equation}
Case (c): 
\begin{equation}\begin{aligned}\label{c}
\Gamma_I\Gamma_\mu\Gamma_J
	& =\epsilon^{I}_{I_0K}\epsilon^{J}_{KJ_0}\Gamma_{I_0}\Gamma_K\Gamma_\mu\Gamma_K\Gamma_{J_0}
	  =(-1)^{\hat J_0}\epsilon^{I}_{I_0K}\epsilon^{J}_{KJ_0}  \Gamma_{I_0}\Gamma_{J_0}\Gamma_K\Gamma_\mu\Gamma_K\\
	&= (-1)^{\hat I_0}(-1)^{\hat K-1}\sigma_K\epsilon^{I}_{I_0K}\epsilon^{J}_{KJ_0}\Gamma_{\mu I_0J_0}\,,
\end{aligned}\end{equation}
where the last equality is due to the following observation: We have  $\Gamma_K\Gamma_\mu\Gamma_K=
(-1)^{\hat K_0}(-1)^{\hat K_0}\Gamma_\mu\Gamma_\mu\Gamma_{K_0}\Gamma_{K_0}\Gamma_\mu
=-\sigma_{K_0}\Gamma_\mu$ and $\sigma_{K}=-(-1)^{\hat K_0}\sigma_{K_0}$  for $K=(\mu,K_0)$. 

\begin{prop}\label{monomial}
Two monomials $c=\alpha\Gamma_{I}$, $d=\beta\Gamma_{J}$ yield a non-degene\-rate quadratic Clifford pair if and only if the index sets $I$ and $J$ coincide up to order or one of both is empty. If $I$ and $J$ coincide up to order and $\hat I>0$ the eigenvalues of $q_{c,d}$ are, up to sign, given by $(\alpha\pm \beta)^2$ with multiplicities $\hat I$ and $\dim V-\hat I$.
\end{prop}

\begin{proof}
If $c,d\neq0$ equations (\ref{a})-(\ref{c}) yield $I_0=J_0 =\emptyset$ for $q_{c,d}$ to be contained in $V$, such that the index sets $I$ and $J$ coincide up to order. 
Therefore, $\q_{c,d}=(\alpha^2-\beta^2)\mathbbm{1}$ if  $\hat K=0$, i.e.\ $I=J=\emptyset$. 
If $\hat K>0$ (\ref{a}) and (\ref{c}) reduce to  
$(-1)^{\hat K}\epsilon^{I}_{J}\sigma_{K}\Gamma_{\mu}$ and 
$-(-1)^{\hat K}\epsilon^{I}_{J}\sigma_K\Gamma_{\mu}$, respectively, such that
\begin{equation}\label{eigenvaluesmono}
q_{c,d}(e_\mu)=\begin{cases} 
\sigma_I\big(\alpha+(-1)^{\hat I}\epsilon_I^J\beta\big)^2\Gamma_\mu&\text{ for }\mu\in I\\
\sigma_I\big(\alpha-(-1)^{\hat I}\epsilon_I^J\beta\big)^2\Gamma_\mu&\text{ for }\mu\not\in I
\end{cases}\,.
\end{equation}
Regardless of the assumption this also includes the case $\alpha\beta=0$.
\end{proof}

\begin{defn}
The pairs from Proposition \ref{monomial} are called {\em monomial quadratic Clifford pairs}.
\end{defn}

\begin{rem}
Consider two non-zero positive numbers $\lambda_1,\lambda_2$. Then for any $I$ with $0<\hat I<\dim(V)$ there exist a monomial quadratic Clifford pair $(c,d)$ such that $q_{c,d}$ obeys (ii) with eigenvalues $\lambda_1,\lambda_2$. The pair is (almost) unique. The mild restriction on uniqueness is due to the  obvious $\mathbbm{Z}_2$-symmetry  determined by the squares in (\ref{eigenvaluesmono}).

It is clear that for every Clifford pair $(c,d)$ and for any non-vanishing scalar $\alpha$ the pair $(\alpha c,\alpha d)$ is a Clifford pair, too. If we consider such pairs as equivalent, then -- in case of monomial quadratic Clifford pairs -- an equivalence class is (almost) uniquely determined by the ratio of the two eigenvalues.  

If we in addition fix the multiplicities of the eigenvalues then there exist at most two pairs associated to $\hat I\leq\hat J$ with $\hat I=\dim(V)-\hat J$. These two cases are related by the volume form $\Gamma_I=\Gamma^*\Gamma_J$. Therefore, if $\dim(V)$ is odd we will restrict to  $\hat I\leq\frac{1}{2}\dim(V)$ because $\Gamma^*=\pm\mathbbm{1}$.
\end{rem}

\subsection{Pseudo monomial solutions}

We consider slightly more general maps $q_{c,d}$ by taking Clifford elements  $c=\alpha\Gamma_I+\beta\Gamma_J$ and $d=\gamma\Gamma_I+ \delta\Gamma_J$ with scalars $\alpha$, $\beta$, $\gamma$, and $\delta$. The calculations in Section \ref{App1} show that the following holds.

\begin{prop}\label{moregen1}  
The approach $c=(\alpha\Gamma_I+\beta\Gamma_J)\Gamma_K$, $d=(\gamma\Gamma_I+\delta\Gamma_J)\Gamma_K$ with non-vanishing scalars $\alpha,\beta,\gamma,\delta$ and $I\cap J=I\cap K=J\cap K=\emptyset$ yields a non-monomial quadratic Clifford pair if and only if the combination is from Table \ref{table1}.

If one of the entries vanishes, the pair $(c,0)$ or $(0,d)$ yields a quadratic Clifford pair if and only if $(-1)^{\hat I\hat J+(\hat I+\hat J)\hat K}=-1$. Then $q_{c,0}=(\alpha^2\sigma_{IK}+\beta^2\sigma_{JK})\mathbbm{1}$, for example.
\end{prop}

\noindent\makebox[\textwidth]{%
\begin{minipage}{1.5\textwidth}\captionof{table}{Some non-monomial quadratic Clifford pairs}\label{table1}
{\centering $
\displaystyle
{
\begin{array}{c|l|l}
\multicolumn{1}{c|}{\text{Case} 	}
			&\multicolumn{1}{c|}{ \text{Conditions on } K,I,J }				
			&\multicolumn{1}{c}{ \text{Pairs} }														\\\hline
\text{1b} 	&\hat I\hat J\neq 0, \hat K=0, (I\cup J)^\complement\neq\emptyset		
			& \big(c,\bar c\big),\  c=\alpha\Gamma_I+\beta\Gamma_J 					\\
			& \hat I\hat J\equiv0\,{\rm mod}\,2		
			&
\\\hline
\text{2a}	& \hat I\hat J\neq 0, \hat K\neq0, (I\cup J\cup K)^\complement=\emptyset, 
			& \big(c,-\bar c\big),\   c=(\alpha\Gamma_K+\beta\Gamma^*)\Gamma_I 		\\
			& \Gamma_J\Gamma_I\Gamma_K=\Gamma^*
			&
	\\\hline
\text{2b} 	&\hat I\hat J\neq 0, \hat K=0, (I\cup J)^\complement=\emptyset
			& \big( c, \pm c\big),\ c=(\alpha+\beta\Gamma^*)\Gamma_I  				\\
			& \Gamma_J=\Gamma^*\Gamma_I, \hat I\equiv\hat J\equiv0\,{\rm mod}\,2	
			&
	\\\cline{2-3}
			& \hat I\hat J\neq 0, \hat K=0, (I\cup J)^\complement=\emptyset	
			& \big(\alpha(\cos(\phi)e^{i\psi}+\sin(\phi)\Gamma^*)\Gamma_I,\\
			& \Gamma_J=\Gamma^*\Gamma_I, \hat I\equiv\hat J\equiv1\,{\rm mod}\,2		
			& \quad  \beta(\cos(\phi)e^{i\psi}-\sin(\phi)\Gamma^*)\Gamma_I \big){}^\dagger			\\\hline
\text{3a} 	&\hat I=0, \hat K\neq0, (J\cup K)^\complement\neq\emptyset
			& \big(c_+,\bar c_-\big),\  c_\pm=(\alpha\pm \beta\Gamma_J)\Gamma_K 	\\			
			& \hat K\hat J\equiv0\,{\rm mod}\, 2
			&
	\\\hline
\text{3b} 	&\hat I=\hat K=0, J^\complement\neq\emptyset
			& \big(\alpha+\beta\Gamma_J,\alpha+\delta\Gamma_J\big)					
	\\\hline
\text{4a} 	&\hat I=0, \hat K\neq0, (J\cup K)^\complement=\emptyset
			& \big(\alpha\Gamma_K+\beta\Gamma^*,\gamma\Gamma_K-\beta\Gamma^*\big)	\\
			& \hat J\equiv\hat K\equiv0\,{\rm mod}\, 2, \Gamma_J\Gamma_K=\Gamma^*
			& \quad  		
	\\\hline			
\text{4b}	&\hat I=\hat K=0, \Gamma_J=\Gamma^*
			& \big(\alpha+\beta\Gamma^*,\alpha+\delta\Gamma^*\big),					\\		
			&\hat J=\dim(V)\equiv0\,{\rm mod}\,2
			&\big(\alpha+\beta\Gamma^*,\gamma-\beta\Gamma^*\big)
	\\
\multicolumn{3}{p{11cm}}{${\quad}^\dagger$  If in case 2b  with $\hat I$ odd both prefactors do not vanish we may write the pair in the form $(c,d)=(\alpha(\gamma+\Gamma^*)\Gamma_I ,\beta(\gamma+\Gamma^*)\Gamma_I$. If we restrict to real or imaginary pairs we may consider $\psi=0$. Finally, there are corresponding monomial pairs, namely for $\phi=0$ or $\phi=\frac{\pi}{2}$.}
\end{array}
}$\\}
\end{minipage}
}

The preceding Proposition tells us what combination yields non-monomial quadratic Clifford pairs. The next natural question is: Which of these is non-degenerate or obeys $q_{c,d}\not\sim\mathbbm{1}$? The answer is also given by the discussion in Section \ref{App1} and is summarized as follows.
\begin{rem}\label{moregen2}
The Clifford pairs of Table \ref{table1} are non-degenerate only in Cases 2b, 3b, 4a, and 4b.
In Case 3b the map $q_{c,d}$ has two eigenvalues that are independent of the parameter $\alpha$. Therefore, we may choose $\alpha=0$ such that we are left with a monomial pair. The same is true for the first pair in Case 4b.\footnote{With respect to Remark \ref{a=0} we could have chosen a priori $\alpha=0$ in the cases with $\hat I=\hat K=0$.}  In Case 4a and the second pair of Case 4b the eigenvalues are independent of $\beta$ such that we may choose $\beta=0$. Moreover, for the pairs in Case 4b we have $q_{c,d}\sim\mathbbm{1}$. 

Last but not least we are left with Case 2b in even dimensions. Here the quadratic form generically admits two eigenvalues. The eigenvalues and their multiplicities in the different subcases are given in the second column of Table \ref{table2}.  

\begin{minipage}{\textwidth}\captionof{table}{Non-monomial, non-degenerate quadratic Clifford pairs in Case 2b}\label{table2}
{\centering $
\displaystyle
{
\begin{array}{l|l}
\multicolumn{1}{c|}{\text{Pair}} &\multicolumn{1}{c}{\text{Eigenvalues}}\\
\hline
	\big((\alpha+\beta\Gamma^*)\Gamma_I,(\alpha+\beta\Gamma^*)\Gamma_I  \big) 	
		& \begin{array}{ll}
				4\sigma_I\alpha^2\,,& \hat I\text{ even} \\ 
				4\sigma_I\beta^2\,,&  \dim(V)-\hat I\text{ even} 
		  \end{array} \\\hline
	\big((\alpha+\beta\Gamma^*)\Gamma_I,-(\alpha+\beta\Gamma^*)\Gamma_I \big) 	
		& \begin{array}{ll}
				4\sigma_I\beta^2\,,& \hat I\text{ even} \\ 
				4\sigma_I\alpha^2\,,& \dim(V)-\hat I\text{ even}
		  \end{array}\\\hline
\begin{array}{l}
		\!\!\big( \alpha(\cos(\phi)+\sin(\phi)\Gamma^*) \Gamma_I, \\
		\qquad \beta(\cos(\phi)-\sin(\phi)\Gamma^*)\Gamma_I \big) 
\end{array}
		& \begin{array}{ll}
				\sigma_I(\alpha-\beta)^2\cos(2\phi)\,,& \hat I\text{ odd} \\ 
				\sigma_I(\alpha+\beta)^2\cos(2\phi)\,,& \dim(V)-\hat I\text{ odd}
		  \end{array}\\
\end{array}
}
$\\}

\end{minipage}
\end{rem}

\begin{defn}\label{def10}
The pairs of Case 2b are called {\em pseudo-monomi\-al qua\-dra\-tic Clifford pairs of even} or {\em odd type} if $\hat I$ is even or odd.
\end{defn} 
\begin{rem}
We will briefly explain why the name pseudo-monomial pair is reasonable in this situation. From Table \ref{table2} we see that the form of the quadratic Clifford pairs is a bit different for $\hat I$ even or odd.
\begin{itemize}
\item For $\hat I$ even we may write $c=\pm d$ in a more symmetric way $c= (a_+\Pi^+ +a_-\Pi^-)\Gamma_I$ with $\alpha_\pm=\alpha\pm\beta$ and $\Pi^\pm=\frac{1}{2}(1\pm\Gamma^*)$ being the projections on the two spaces of half spinors. Then the two eigenvalues of $q_{e,f}$ again arise as the squares of difference and sum of the two coefficients. 
\item For $\hat I$ odd we consider $\sin(\phi)\neq0$ and write $\gamma=\cot(\phi)$ as well as $\bar\alpha=\sin(\phi)\alpha$ and $\bar\beta=\sin(\phi)\beta$. We may write $\gamma+\Gamma^*=(\gamma+\epsilon)-2\epsilon\Pi^{-\epsilon} $ with $\epsilon=\pm1$. Then $c=\bar\alpha(\gamma+\epsilon)\Gamma_I-2\bar\alpha\epsilon\Gamma_I\Pi^{+\epsilon}$ and $d=\bar\beta(\gamma+\epsilon)\Gamma_I-2\bar\beta\epsilon\Gamma_I\Pi^{-\epsilon}$. We insert this in (\ref{sum}) and see $q_{\bar\alpha\Gamma_I\Pi^{+\epsilon},\bar\beta\Gamma_I\Pi^{-\epsilon}}(v)=0$ as well as 
$\big\{s_{\bar\alpha\Gamma_I,\bar\beta\Gamma_I}\,,\, s_{\bar\alpha\Gamma_I\Pi^{+\epsilon},\bar\beta\Gamma_I\Pi^{-\epsilon}}\big\}=q_{\bar\alpha\Gamma_I,\bar\beta\Gamma_I}$ such that 
\begin{equation*}
q_{c,d}=(\gamma+\epsilon)^2 q_{\bar\alpha\Gamma_I,\bar\beta\Gamma_I}
		-2\epsilon(\gamma+\epsilon)q_{\bar\alpha\Gamma_I,\bar\beta\Gamma_I}
	   =(\gamma^2-\epsilon^2)q_{\bar\alpha\Gamma_I,\bar\beta\Gamma_I}\,.
\end{equation*}
Again, the eigenvalues are sum and difference of the two coefficients of the equivalent monomial pair $(c',d')=(\sqrt{\cos(2\phi)}\,\alpha\Gamma_I,\sqrt{\cos(2\phi)}\,\beta\Gamma_I)$.
\end{itemize}
\end{rem}
Proposition \ref{monomial} and Remark \ref{moregen2} yield the following proposition.
\begin{prop}\label{monomial-th}
Let $B$ be a non-degenerate symmetric linear map on a Euclidean vector space $V$ that has two different non-zero eigenvalues.
If $\dim(V)$ is odd then there exists one monomial quadratic Clifford pair with $q_{c,d}=B$. 
If $\dim(V)$ is even  then there exist two monomial quadratic Clifford pairs as well as a family of  pseudo-monomial quadratic Clifford pair(s). The family is parameterized by $\{\pm 1\}$ or $S^1$ if the multiplicity of each eigenvalue is even or odd, respectively. 
In any case there is a further discrete symmetry coming from the squares in (\ref{eigenvaluesmono}) and Table \ref{table2}.
\end{prop}

\subsection{Linear solutions}

The second pair of Case 4b in Table \ref{table1} has an additional property: The square root $s_{c,d}$ of $q_{c,d}$ obeys $s_{c,d}\big|_{V}\subset V$. 
In the sense of \cite{santi1} the quadratic Clifford pairs $(c,d)$ for which the restriction to $V$ of the square root of $q_{c,d}$  is a map into $V$  are called linear pairs. 
We will relax this notion of linearity in the following definition and we will 
examine the pairs with this property. This yields a classification of linear quadratic Clifford pairs.
\begin{defn}
Consider the square root of $q_{a,b}$ given by $s_{a,b}:x\mapsto ax-xb$. We call a pair $(a,b)\in \cl(V)$ {\em linear} if 
$s_{a,b}$ restricted to $V\otimes \mathbbm{C}$ has image in $V\otimes \mathbbm{C}$.
\end{defn}
\begin{lem}\label{root}
The restriction of $s_{a,b}:x\mapsto ax-xb$ to $V\otimes \mathbbm{C}$ is a map into $V\otimes \mathbb{C}$ if and only if there exists scalars $\alpha_0,\alpha_1,\beta_0$ and $A\in \cl_2(V)$ such that $a=\alpha_0+A+\alpha_1\Gamma^*$ and $b=\beta_0+A-\alpha_1\Gamma^*$ if $\dim(V)$ is even, or $a=\alpha_0+A$ and $b=\beta_0+A$ if $\dim(V)$ is odd. 
\end{lem}
\begin{rem}
$s_{a,b}\big|_V$ in Lemma \ref{root} is independent of the terms proportional to $\Gamma^*$. Therefore, we may consider $\alpha_1=0$ when we are only interested in this restriction. Furthermore, we may consider $\alpha_0=0$ because of $s_{a,b}=s_{\alpha+a,\alpha+b}$ for all scalars $\alpha$. 
\end{rem}
\begin{proof}
Consider $a=\sum_Ia_I\Gamma_I$ and $b=\sum_Ib_I\Gamma_I$ and define $\theta_I^\mu=1$ if $\mu\not\in I$ and $\theta^\mu_I=-1$ if $\mu\in I$. Then  $s_{a,b}\Gamma_\mu=\sum_I(a_I-\theta^\mu_I(-1)^{\hat I} b_I)\Gamma_I\Gamma_\mu$ and we write\footnote{$\eta^\mu_I$ denotes the sign of the permutation that indicates the place of $\mu$ within the multi-index $I$.}
\[
s_{a,b}\Gamma_\mu=\sum_{\mu\in I}\eta_I^\mu(a_I+(-1)^{\hat I}b_I)\Gamma_{I\setminus \mu}
				 +\sum_{\mu\not\in I}(a_I-(-1)^{\hat I}b_I)\Gamma_{I\cup\mu}\,.
\]
The condition $s_{a,b}\Gamma_\mu\in V$ for all $\mu$ is equivalent to
(a) $a_I+(-1)^{\hat I}b_I=0$ for all $I$ such that there exists a $\mu$ with $\mu\in I$ and $\widehat{I\setminus\mu}\neq 1$ and 
(b) $a_I-(-1)^{\hat I}b_I=0$ for all $I$ such that there exists a $\mu$ with $\mu\not\in I$ and $\widehat{I\cup\mu}\neq 1$. 

If $\dim(V)= 2m$ is even then this is the same as
(a) $a_I+(-1)^{\hat I}b_I=0$ for all $I$ with $\hat I\neq 0,2$  and 
(b) $a_I-(-1)^{\hat I}b_I=0$ for all $I$ such $\hat I\neq 0,2m$.
Therefore, 
$a_I=b_I=0$ for all $\hat I\neq 0,2,2m$. Furthermore, there is no restriction for $\hat I=0$, i.e.\ for the scalar part. For $\hat I=2$ the coefficients have to obey $a_{(2)}=b_{(2)}$ and for $\hat I=2m$ the  coefficients must obey $a_{(2m)}=-b_{(2m)}$. 
For $\dim(V)=2m+1$ we only need to consider $\hat I\leq m$ due to $\Gamma^*\sim\mathbbm{1}$ such that the same calculations as before yield $a_I=b_I=0$ for $\hat I\neq 0,2$, no restriction for $\hat I=0$, and  $a_{(2)}=b_{(2)}$.
\end{proof}

If we identify $\cl_2(V)$ in the usual way with the skew-symmetric endomorphism of $V\otimes \mathbbm{C}$, i.e.
\begin{equation}
(A_{\mu\nu}) \mapsto -\frac{1}{4}\sum_{\mu,\nu}A_{\mu\nu}\Gamma^{\mu\nu}\,,
\end{equation} 
the action of $s_{A,A}$ on $x\in V\otimes \CC$ is just the action of the endomorphism $A$ on $x$. 
We write $s_{A,A}(x)=Ax-xA=[A,x]=A(x)$ in this special case. In particular, the restriction of $s_{A,A}$ to $V\otimes \mathbbm{C}$ is skew-symmetric. 

\begin{prop}\label{prop2}
A quadratic Clifford pair is linear if and only if it is of the form $(0,\beta)$ with a scalar $\beta$ or $(A,A)$ with $A=A_0+iA_1\in \cl_2(V)$ such that $A_0$ and $A_1$ anti-commute when considered as endomorphisms of $V$.
\end{prop}
\begin{proof}
Consider a linear pair $(a,b)$ given by $a=A$ and $b=\beta+A$ this yields
$q_{a,b}(x)=\beta x + q_{A,A}(x) + \beta [A,x]$. This is symmetric if and only if $\beta=0$ or $A=0$.
Moreover, for $A=0$ condition (i) is fulfilled for $\beta$ real. Now consider $\beta=0$ such that $a=b=A_0+iA_1$ is a complex skew-symmetric endomorphism. Then $q_{A,A}$, i.e.\ the endomorphism $A^2=(A_0^2-A_1^2)+i\{A_0,A_1\}$, preserves the real vector space $V$ if and only if the real skew-symmetric endomorphisms $A_0$ and $A_1$ anti-commute. For a characterization of skew-symmetric endomorphisms with this special feature see \cite{Putt}, for example.
\end{proof}

\begin{rem}
Two particular cases in Proposition \ref{prop2} are provided by $A$ skew-symmetric and real or skew-symmetric and purely imaginary. In this situation $q_{A,A}$ is negative or positive (semi-)definite. In particular, all eigenvalues come in pairs such that in the odd dimensional case $q_{A,A}$ is always degenerate.
\end{rem}

\begin{coro}\label{linear}
Given a symmetric linear map $B$ on the vector space $V$ such that all nonzero eigenvalues of $B$ have even multiplicities. Then there exists $A\in  \cl_2(V)$ such that $(A,A)$ yields a linear quadratic Clifford pair associated to $B$.
\end{coro}
\begin{proof}
We consider an orthonormal basis $\mathcal{W}$ of $(V,g)$ with induced isomorphism $\Phi_\mathcal{W}:V\to\mathbbm{R}^n$ such that $\Phi_\mathcal{W}\circ B\circ\Phi_\mathcal{W}^{-1}=\Delta^2$ is a diagonal matrix. The square root $\Delta$ is defined by $\Delta=\Delta_0 + i\Delta_1$ with 
$\Delta_0={\rm diag}\left(\Sigma_1,\ldots,\Sigma_r,\mathbf{0}_{n-2r}\right)$ and 
$\Delta_1={\rm diag}\left(\mathbf{0}_{2r},\Sigma_{r+1}, \ldots,\Sigma_{r+s},\mathbf{0}_{n-2r-2s}\right)$ where  
$\Sigma_j={\small\renewcommand{\arraystretch}{1.0}
	\begin{pmatrix}0&\lambda_j\\ -\lambda_j&0\end{pmatrix}}$ 
for $1\leq j\leq r$ and 
$\Sigma_{r+k}={\small\renewcommand{\arraystretch}{1.0}
	\begin{pmatrix}0&\kappa_k\\-\kappa_k&0\end{pmatrix}}$ 
for $1\leq k\leq s$ and $\lambda_j,\kappa_k>0$. 
Let $A:=\Phi_\mathcal{W}^{-1}\Delta\Phi_\mathcal{W}$ and identify this skew-symmetric map as usual with $A\in \cl_2(V)$. This element obeys $q_{A,A}=B$.
\end{proof}

\begin{rem}
Let $(a,b)$  be a quadratic Clifford pair with $a=a'$ and $b=\beta+b'$ and scalar part $\beta\neq0$. Then either $(a',b')$ is not a quadratic Clifford pair or it doesn't contribute to $q_{a,b}\big|_V$. 

Suppose $(a',b')$ is a quadratic Clifford pair and recall $q_{a,b}(x) = \beta^2x +q_{a',b'}(x) -2\beta\,s_{a',b'}(x)$. From Lemma \ref{root} we know that $s_{a',b'}$ is a self-map of $V$ if and only only if  $a'=b'=A$ if $\dim(V)$ is odd or $a'=A+\alpha_1\Gamma^*$ and $b'=A-\alpha_1\Gamma^*$ if $\dim(V)$ is even. In this situation $s_{a',b'}\big|_V=s_{A,A}\big|_V$ is skew-symmetric, such that $A=0$.
\end{rem}

\begin{coro}\label{Dim3}
In dimension $n=3$ any symmetric linear map that can be realized by quadratic Clifford pairs is either proportional to $\mathbbm{1}$ or degenerate with a non-zero eigenvalue of multiplicity two. 
\end{coro}

\subsection{Generalized monomial solutions}

We generalize the results obtained so far and show that we can realize symmetric maps $B$ with more than two eigenvalues. The proof of the following proposition is given in Section \ref{App1}, too.
\begin{prop}\label{hom}
Consider $c=\sum_\alpha c_\alpha\Gamma_{I_\alpha}$ and $d=\sum_\alpha d_\alpha\Gamma_{I_\alpha}$ with $d_\alpha=\ep_\alpha c_\alpha$ for $\ep_\alpha\in\{\pm1\}$, $I_\alpha\cap I_\beta=\emptyset$ for $\alpha\neq\beta$, and 
$\big(\bigcup_\alpha I_\alpha\big)^\complement=\emptyset$. Furthermore, let the amount of summands in $c$ be bigger than two.
Then $(c,d)$ is a non-degenerate quadratic Clifford pair if and only if $\hat I_\alpha$ is even for all $\alpha$ in even dimension and $\hat I_\alpha$ is even up to one exception in odd dimension. Furthermore we have  $\ep_\alpha=(-1)^{\hat I_\alpha}$. 
In this situation the eigenvalues of $q_{c,d}$ are given by $4\sigma_{I_\alpha}c_\alpha^2$ with multiplicities $\hat I_\alpha$.
\end{prop}

\begin{rem}
\begin{itemize}
\item
Suppose $c$ and $d$ decompose into two summands such that the subcases we denoted by a) in Table \ref{table:app} do not appear. Then the conditions on the signs that describe $d$ are less restrictive and we recover the pseudo-monomial solutions of Table \ref{table2}.
\item
Because there are no restrictions to the values of $c_\alpha$, some of them may coincide. So the result in Proposition \ref{monomial-th} is extended by generalized monomial pairs with exactly two different parameters.
\item
For $\hat I_\alpha=2$ for all $\alpha$ we have $c=d\in\mathfrak{so}(V)$ such that we are left with a linear pair described in Proposition \ref{prop2}.
\item
For $\hat I$ even we have $[\Gamma_I,\Gamma_\mu]=2k\delta_{\mu[\nu_1}\Gamma_{\nu_2\ldots\nu_{\hat I}]}=2\Gamma_\mu\lfloor \Gamma_I$.  Therefore, in situation of Proposition \ref{hom} the square root of $q_{c,d}$ is given by 
$
s_{c,c}(v)=2 v\lfloor c
$.
\end{itemize}
\end{rem}
We may further generalize the above result in the even dimensional case and allow pseudo-monomial summands in the above construction. 
\begin{prop}\label{hom-p}
Consider $c,d$ with 
\[
c=\sum_\alpha c_\alpha\Gamma_{I_\alpha}+\Gamma^*\sum_\alpha\hat c_\alpha\Gamma_{I_\alpha}
\text{ and }
d=\sum_\alpha \ep_\alpha c_\alpha\Gamma_{I_\alpha}+\Gamma^*\sum_\alpha\hat\ep_\alpha\hat c_\alpha\Gamma_{I_\alpha}
\]
and  $\ep_\alpha=-\hat\ep_\alpha=(-1)^{\hat I_\alpha}$. Then $(c,d)$ yields a non-degenerate quadratic Clifford pair if and only if it is of one of the following types:
\begin{enumerate}
\item\label{item:x}   
All summands are even with $\hat c_\alpha c_\alpha=0$ for all $\alpha$, i.e.\ only one of $\Gamma_{I_\alpha}$ and $\Gamma^*\Gamma_{I_\alpha}$ contributes to $c$ and $d$.  
\item\label{item:y}   
Two summands are odd, e.g.\ $\hat I_{\alpha_0},\hat I_{\alpha_1}$, and all others obey $\hat I_\alpha$ even with $\hat c_\alpha=0$ as well as $c_{\alpha_0}c_{\alpha_1}=\hat c_{\alpha_0}\hat c_{\alpha_1}$. 
For example
\begin{align*} 
c&=c_{\alpha_0}(\gamma+\Gamma^*)\Gamma_{I_{\alpha_0}} 
+c_{\alpha_1}(1+\gamma\Gamma^*)\Gamma_{I_{\alpha_1}} 
+ \sum_{\hat I_\alpha\,\text{even}} c_\alpha\Gamma_{I_\alpha}\,,\\
d&=-c_{\alpha_0}(\gamma-\Gamma^*)\Gamma_{I_{\alpha_0}} 
-c_{\alpha_1}(1-\gamma\Gamma^*)\Gamma_{I_{\alpha_1}} 
+ \sum_{\hat I_\alpha\,\text{even}} c_\alpha\Gamma_{I_\alpha}
\end{align*}
if $\hat c_{\alpha_0}, c_{\alpha_1}\neq0$.
\end{enumerate}
\end{prop}
\begin{rem}
Regardless of whether the choice of signs $\ep_\alpha,\hat\ep_\alpha$ we use here is unique, we will see later that this choice is natural. Nevertheless, in Section \ref{App1} we list the general conditions such that a pair that satisfies the approach used in Proposition \ref{hom-p} without the sign fixing yields a quadratic Clifford pair. 
\end{rem}

\begin{defn}\label{gen-mon}
The quadratic Clifford pairs of Propositions \ref{hom} and \ref{hom-p} are called {\em generalized monomial}. 
\end{defn}

\section{Connections on Cahen-Wallach spaces}
\subsection{Cahen-Wallach spaces}\label{solvable}

Solvable Lorentzian symmetric spaces -- short CW-spaces -- have been classified by M.\ Cahen and N.\ Wallach in \cite{CW70}. They discovered a one-to-one correspondence between CW-spaces and triples $(V,\langle\cdot,\cdot\rangle,B)$.  Here $V$ is an $n$-dimensional real vector space, $\langle\cdot,\cdot\rangle$ is a block-diagonal Lorentzian metric on the extension $W:=V\oplus\RR^2$ such that $g:=\langle\cdot,\cdot\rangle\big|_V$ is positive definite, and $B\in {\rm End}(V)$ is symmetric with respect to $g$. 

We shortly recall this correspondence. We choose a null-basis $e_+,e_-$ of the factor $\RR^2$ of $W$, i.e.\ $\langle e_+,e_+\rangle=\langle e_-,e_-\rangle=0$, $\langle e_+,e_-\rangle=1$. Therefore, we may identify $\RR_+=\RR e_+=\RR$ and $\RR e_-=\RR_-=\RR^*$, i.e. $(e_+)^*=e_-$ with $*:W\to W^*$ being the isomorphism $w\mapsto \langle w,\cdot\rangle$. 

The space $\mathfrak{g}=V^*\oplus W =V^*\oplus (V\oplus \RR_+ \oplus \RR_-)$ is a Lie algebra subject to the  commutation relations
\begin{align}
\left[v^*,w\right] &= -v^*(Bw)\cdot e_+=  -\langle Bv,w \rangle\cdot e_+\,, \label{c1}\\
\left[v^*, e_-\right] &= Bv\,, \label{c2}\\
\left[e_-,w\right] &= w^*\,, \label{c3}
\end{align}
and all other combinations vanishing. In particular, $\RR_+$ is the one dimensional center of $\mathfrak{g}$.
Within $\mathfrak{g}$ the factor $V^*$ acts on $W$, $[W,W]=V^*$, and $\langle\cdot,\cdot\rangle$ is $V^*$-invariant. From (\ref{c1})-(\ref{c3}) we see that the embedding 
\begin{equation}
V^*\longrightarrow \RR_+\otimes V \hookrightarrow \mathfrak{so}(W)
=\mathfrak{so}(V)\,\oplus\, (\RR_+\otimes V) \,\oplus\,(\RR_-\otimes V)\,\oplus\,( \RR_+\otimes\RR_-)
\end{equation}
is given by $v^*\mapsto  Bv\wedge e_+$ where $x\wedge y (z):=\langle y,z\rangle x-\langle x,z\rangle y$. These data provide a symmetric space with Lorentz metric determined by $\langle\cdot,\cdot\rangle$ and $B$.
The resulting Lorentzian space $M_B$ is indecomposable if and only if the symmetric map $B$ is non-degenerate. This can best be seen if we recall that $M_B$ is decomposable if there exists a $V^*$-invariant subspace $\tilde W\subset W$ such that $\langle\cdot,\cdot\rangle|_{\tilde W\times \tilde W}$ is non-degenerate, see \cite{Baum12, Wu67}.
Moreover, two Lorentzian spaces defined in this way by symmetric maps $B_1$ and $B_2$ are isometric, if and only if $B_1$ and $B_2$ are conformally equivalent, i.e.\ there exists a real scalar $c>0$ and an orthogonal transformation $O$ such that $B_2=c O^tB_1 O$. 

Mainly to fix our notation, we recall some facts on the Clifford algebra in this special situation. We consider $\cl(\mathbbm{R}^{1,1})=\mathfrak{gl}_2\CC$ with generators 
$\Gamma_+=\Gamma(e_+)=\frac{1}{\sqrt{2}}(i\sigma_2+\sigma_1)
		 =\sqrt{2}{\renewcommand{\arraystretch}{1.0}\begin{pmatrix}0&1\\0&0\end{pmatrix}}$, $
 \Gamma_-=\Gamma(e_-)=\frac{1}{\sqrt{2}}(i\sigma_2-\sigma_1)
 		 =\sqrt{2}{\renewcommand{\arraystretch}{1.0}\begin{pmatrix}0&0\\-1&0\end{pmatrix}}$ 
and we denote the two-dimensional volume element by 
$\sigma:=\frac{1}{2}\left[\Gamma_+,\Gamma_-\right]=-\sigma_3
		={\renewcommand{\arraystretch}{1.0}\begin{pmatrix}-1&0\\0&1\end{pmatrix}}$. 
If we denote the generators of $\cl(V)$ by $\{\Gamma_\mu\}_{1\leq \mu\leq n}$ those of 
$\cl(W)=\mathfrak{gl}_2\CC \,\hat\otimes \,\cl(V)$  are given by 
$\{\Gamma_+\otimes\ei, \Gamma_{-}\otimes\mathbbm{1},\sigma\otimes\Gamma_\mu\}$. 
In particular, 
\begin{align}
\mathfrak{gl}_2\CC \ni r &\mapsto r\,\hat\otimes\,\ei = r \otimes\mathbbm{1} \in \cl(W)\label{incl-gl}\,,\\
\cl(V)\ni a &\mapsto \mathbbm{1}\,\hat\otimes\,a=\mathbbm{1}\otimes a^0+ \sigma\otimes a^1\in \cl(W)\,.\label{incl-cl}
\end{align}
Here $a=a^0+a^1\in \cl(V)$ is the decomposition in even and odd part defined by $\bar{\ }:\cl(V)\to \cl(V)$ according to Remark \ref{aut}.
Consider the irreducible Clifford modules $S_2$ and $S(V)$ of $\cl(\mathbbm{R}^{1,1})$ and $\cl(V)$ whose elements are called spinors. The first one decomposes into a sum of two one dimensional half spinor spaces $S_2^\pm={\rm ker}(\Gamma_\mp)$ given by the  $\pm1$-eigenspaces of $\sigma$. If we denote the two projections on the two eigenspaces by $\sigma_\pm=\frac{1}{2}(\ei\pm \sigma)$ then (\ref{incl-cl}) is rewritten as 
$\ei\,\hat\otimes\, a=\sigma_-\,\hat\otimes\,a+\sigma_+\hat \otimes\,a=\sigma_-\otimes \bar a+\sigma_+\otimes a$.
In our choice of $\gamma$-matrices the eigendirections are given by $\vec{e}_1=(1,0)^t$ and $\vec{e}_2=(0,1)^t$ such that a spinor in $S(W)= S_2\,\hat\otimes\,S(V)$ can be written as 
$\vec\eta={\begin{pmatrix}\eta_1\\\eta_2\end{pmatrix}}= \vec{e}_1\otimes\eta_1+ \vec{e}_2\otimes \eta_2$. 
The action of $\cl(W)$ on $S(W)$ is now given by 
\begin{equation}\label{matr}\begin{aligned}
{(r \,\hat\otimes\,a)\begin{pmatrix}\eta_1\\\eta_2 \end{pmatrix} 
=\begin{pmatrix}r_{11}\bar a & r_{12}a\\r_{21} \bar a & r_{22}a\end{pmatrix}
	\begin{pmatrix}\eta_1\\\eta_2 \end{pmatrix} 
= r\begin{pmatrix}\bar a\eta_1\\ a\eta_2
\end{pmatrix}}
\end{aligned}
\end{equation}
for $r\in\mathfrak{gl}_2\CC$ and $a\in\cl(V)$. 
In particular, the image of $v^*\in V^*$  is given by 
\begin{equation}
\begin{aligned}
v^*= Bv\wedge e_+ \mapsto \ &
	\frac{1}{4}\big(( \Gamma_+\otimes\ei)( \sigma\otimes Bv)-(\sigma\otimes Bv)( \Gamma_+ \otimes\ei)\big) \\
&=	\frac{1}{4} (\Gamma_+\sigma-\sigma\Gamma_+)\otimes  Bv 
= \frac{1}{2}\Gamma_+\otimes Bv
= \frac{1}{\sqrt{2}}{\begin{pmatrix}0&Bv\\0&0\end{pmatrix}}
\end{aligned}
\end{equation}
when considered as an element of $\mathfrak{so}(W)\subset \cl(W)$ under the spin representation.

\subsection{Invariant connections and Clifford maps}

We follow here the explanations in \cite[ch.~X]{KN2} and \cite[ch.~1]{CS} and we cordially refer the reader to these texts for more details.
Consider a reductive spin homogeneous space $G/H$, i.e.\ the Lie algebras $\mathfrak{g}$ and $\mathfrak{h}$ of $G$ and $H$ are complemented by an $ad_\mathfrak{h}$-invariant subspace $\mathfrak{p}$ such that $\mathfrak{g}=\mathfrak{h}\oplus\mathfrak{p}$. 
We identify $\mathfrak{h}$ with its image in $\mathfrak{so}(\mathfrak{p})$ via the isotropy representation. Furthermore, let $\mathfrak{h}'\subset\mathfrak{so}(\mathfrak{p})$ be a subalgebra such that $[\mathfrak{h}',\mathfrak{h}]\subset\mathfrak{h}$. In particular, $\mathfrak{h}'$ is the Lie algebra to a sub group $H'$ of $N_H$ where $N_H\subset SO(\mathfrak{p})$ is the normalizer of $H$ in $SO(\mathfrak{p})$. Under mild conditions on $H'$ we may assume $G/H\simeq \bar G/\bar H$ with $Lie(\bar G)=\mathfrak{g}\oplus\mathfrak{h}'$, $Lie(\bar H)=\mathfrak{h}\oplus\mathfrak{h}'$, see \cite[7.B]{Besse}.

Let $G/H$ be a reductive homogeneous space with decomposition $\mathfrak{g}=\mathfrak{h}\oplus\mathfrak{p}$. 
A  principal bundle $K\to P\to G/H$ over $G/H$ is called homogeneous if there exist a smooth homomorphism $\imath:H\to K$ such that $P\simeq G\times_H K$ via the left action of $H$ on $K$. For example $K=H$ and $\imath=id$ yields the $H$-principal bundle $H\to G\to G/H$.
The $G$-invariant connections on $P$ are in one-to-one correspondence with linear maps $\Lambda:\mathfrak{g}\to\mathfrak{k}$ such that 
$\Lambda(h)=\imath_*(h)$ and 
$\Lambda\circ {\rm ad}_h={\rm ad}_{\imath_*(h)}\circ \Lambda$ for all $h\in\mathfrak{h}$. 
This correspondence is taken over to associated bundles in the following manner. 
Consider a  representation $\beta: K\to GL(V)$  and let $E=P\times_\beta V=G\times_{\beta\circ \imath} V$ be the vector bundle associated to $P$. Then the $G$-invariant linear connections on $E$ are in one-to-one correspondence with linear maps $\Lambda:\mathfrak{g}\to \mathfrak{gl}(V)$ such that 
$\Lambda(h)=\beta_*\circ \imath_*(h)$ and 
$\Lambda([h,x])=[\beta_*\circ\imath_*(h),\Lambda(x)]$ for all $h\in\mathfrak{h},x\in\mathfrak{g}$.

The curvature of the connection that is defined by $\Lambda$ is given by
\begin{equation}\label{curv}
\mathcal{R}^{\Lambda}(x,y)=\big[\Lambda(x),\Lambda(y)]-\Lambda([x,y]_\mathfrak{p})-\beta_*\circ\imath_*([x,y]_\mathfrak{h})\,.
\end{equation}

We will apply the above to the case of CW-spaces and take over the notations from Section \ref{solvable}. We consider $G$ with $\mathfrak{g}=V^*\oplus W$ and a vector bundle associated to $G$ via the spin representation $\Gamma:\mathfrak{spin}(W)\supset V^* \oplus \mathfrak{h}'\to \cl(W)$. Here $\mathfrak{h}'\subset\mathfrak{so}_B(V)$ where 
\begin{equation}
\begin{aligned}
\mathfrak{so}_B(V)&=\{A\in\mathfrak{so}(V)\,|\, [A,B]=0\}\\
& =\{A\in\mathfrak{so}(V)\,|\, [A,v^*]=(Av)^*\}\subset\mathfrak{so}(V)\subset\mathfrak{so}(W)
\end{aligned}
\end{equation}
denotes the subalgebra that leaves invariant the symmetric map $B$ that defines the CW-space. In particular the isometry algebra of $M_B$ obeys
\[
\mathfrak{isom}(M_B)=\mathfrak{so}_B(V)\oplus V^*\oplus W\,,
\]
see \cite{CW70, Neukirchner}.

Let us denote by $S$ the spinor module on which $\cl(W)$ acts, i.e.\ $\cl(W)\subset  \mathfrak{gl}(S)$ and by $D$ an invariant connection of the associated bundle $\cancel{S}=G\times_\Gamma S$ defined by the equivariant map $\Lambda$. Then the space of parallel sections $\mathcal{K}_1:=\{\psi\in {\rm sec}(\cancel{S})\,|\,D\psi=0\}$ is isomorphic to the subspace $S'\subset S$ given by $S':=\{\xi\in S\,|\, R^\Lambda(x,y)\xi=0 \text{ for all }x,y\in V^*\oplus W\}$, see \cite{BaumKath99}.

To emphasize the algebraic character of invariant connections expressed by the defining equivariant linear map we introduce the following notions.

\begin{defn}\label{Cliffordmap}
Let $M_B$ be a CW-space defined by the symmetric map $B$ and  $\mathfrak{h}'\subset\mathfrak{so}_B(V)$ be a subalgebra.
\begin{enumerate} 
\item
A {\em Clifford map} of $(M_B,\mathfrak{h}')$ 
is a $V^*\oplus\mathfrak{h}'$-equivariant linear map 
$\rho:\mathfrak{h}'\oplus V^*\oplus W \to\cl(W)$. Moreover, the image of the Clifford map $\rho$ of $(M_B,\mathfrak{h}')$ acts on a Clifford module $S$ and 
the restriction to $V^*\oplus\mathfrak{h}'$ is fixed to be the spin-representation, i.e.\  $\rho\big|_{V^*\oplus\mathfrak{h}'}=\Gamma$. 
\item 
A Clifford map $\rho$ is called {\em Clifford representation} if it is a representation of  $\mathfrak{h}'\oplus V^*\oplus W$ on the Clifford module $S$.
\end{enumerate}
\end{defn}

\begin{nota}
A Clifford map  $\rho$ is called {\em simple} if $\mathfrak{h}'=0$. It is called {\em irreducible} if $S$ is, otherwise it is called {\em reducible}.
\end{nota}

\begin{rem}\label{rem:restrKill}
\begin{itemize}
\item
If we consider a Clifford map as invariant connection the Clifford representations are exactly those connections for which the curvature vanishes. Therefore, the curvature of a Clifford map $\rho$ measures the defect from being a Clifford representation.
\item In case of a Clifford representation the space $\mathcal{K}_1$ of parallel sections has maximal dimension $\dim(S)$.
\item It is also interesting to consider subspaces $S'\subset S$ and Clifford maps for which the Clifford map restricted to the subspace is indeed a representation. They are associated to connection with a reduced amount of parallel sections, namely $\dim(\mathcal{K}_1)=\dim(S')<\dim(S)$. We will come back to this in Section \ref{restr}.
\end{itemize}
\end{rem}

\subsection{Simple irreducible Clifford representation of $M_B$}\label{sec:simple-irred}

We consider a CW-space $M_B$ with associated Lie algebra $\mathfrak{g}=V^*\oplus W=V^*\oplus(V\oplus\RR_-\oplus\RR_+)$ as before. An irreducible simple Clifford map of $M_B$ is a $V^*$-equivariant map $\rho:\mathfrak{g}\to\cl(W)$ that acts on the irreducible Clifford module $S=S(W)=S_2\,\hat\otimes\,S(V)$, see Definition \ref{Cliffordmap}
Moreover, the restriction $\rho\big|_{V^*}$ is assumed to be given by the spin representation $\Gamma: V^*\subset\mathfrak{so}(W) \to \cl(W)$, $\Gamma(v^*)=\frac{1}{2}\Gamma_+\,\hat\otimes\,Bv$.

An examination of the involved brackets regarding to the $V^*$-equivariance yields that a Clifford map $\rho$ of $\mathfrak{g}$ on $S$ is of the form 
\begin{align}
\rho(v^*) &	=\frac{1}{2}\Gamma_+\,\hat\otimes\, Bv 	\,,						\label{19}\\
\rho(e_+) &	= \Gamma_+\,\hat\otimes\, a \,,									\label{20}\\
\rho(e_-) &	= \sigma_-\,\hat\otimes\, c+\sigma_+\,\hat\otimes\, d 
				+\Gamma_+\,\hat\otimes\,e-\Gamma_-\,\hat\otimes\, b \,,			\label{21}\\
\rho(w)   &	= -\sigma_-\,\hat\otimes\, w b-\sigma_+\,\hat\otimes\,\bar bw
				-\frac{1}{2}\Gamma_+\,\hat\otimes\,s_{\bar c,d}(w)	\,,			  	 	\label{22}
\end{align}
for $a,b,c,d\in\cl(V)$, see Section \ref{app3}. In terms of matrices as in (\ref{matr}) this is 
\begin{gather*}
\rho(v^*) =  \frac{1}{\sqrt{2}}{\begin{pmatrix}0&Bv\\0&0\end{pmatrix}}\,,\qquad
\rho(e_+) = {\begin{pmatrix}0& \sqrt{2}\,a\\ 0 &0\end{pmatrix}}\,, \\
\rho(e_-) = {\begin{pmatrix}\bar c & \sqrt{2}\, e \\ \sqrt{2}\,\bar b & d\end{pmatrix}}\,,\qquad
\rho(w)   = {\begin{pmatrix}w\bar b & -\frac{1}{\sqrt{2}}s_{\bar c,d}(w)\\ 0 &-\bar bw\end{pmatrix}}\,.
\end{gather*}
An additional non-trivial condition that is a consequence of the $(V^*,\RR_-)$-bracket is
\begin{equation}\label{24}
	\Gamma_{\{\mu} \bar b \Gamma_{\nu\}}  = g_{\mu\nu} a \,.
\end{equation}
If we consider (\ref{24}) as a bilinear expression on $V$ and take the trace with respect to the metric on $V$ we get 
\begin{equation}\label{24a}
a =  \frac{1}{n}\sum_{\mu=1}^n\Gamma^\mu \bar b\Gamma_\mu\,.
\end{equation}
If we again use the images $\Gamma_\mu$ of an orthonormal basis we get for the expansion of $\bar b=\sum_{I}b_I\Gamma_I$ and for all $\mu\neq\nu$ 
\begin{align*}
0 
 =    2\sum_{I}  b_I\Gamma_{\{\mu}\Gamma_I\Gamma_{\nu\}}
 =    2\!\!\!\sum_{I:\mu\not\in I,\nu\in I}\!\!\!(-1)^{I} b_I\Gamma_I\Gamma_{\mu\nu}
    - 2\!\!\!\sum_{I:\mu\in I,\nu\not\in I}\!\!\!(-1)^{I} b_I\Gamma_I\Gamma_{\mu\nu}\,.
\end{align*}
Therefore, $ b_I=0$ for all $I$ such that there exist $\mu,\nu$ with $\mu\in I$ and $\nu\not\in I$, i.e. for all $I$ with $\hat I\neq 0,n$. If $n$ is odd this yields that $\bar b=\alpha$ is a scalar and if $n$ is even we have $\bar b=\alpha+\beta \Gamma^*$. 
Then $a$ is determined by $\bar b$ because (\ref{24a}) yields $a=-\alpha$ if $n$ is odd and  $a=-(\alpha-\beta\Gamma^*)$ if $n$ is even.
At this point we found the simple Clifford maps of $M_B$:

\begin{prop}\label{maps}
A simple Clifford map of $M_B$ is given by (\ref{19})-(\ref{22}) with $a=\alpha+\beta\Gamma^*$ and $b= -\alpha+\beta\Gamma^*$ if $\dim(V)$ is even, and $a=-b= \alpha$ if $\dim(V)$ is odd.
\end{prop}

\begin{rem}
As noticed in (\ref{curv}) the $(W,W)$-brackets of a Clifford map interpreted as an invariant connection are connected to its curvature. In the case that $\dim(V)$ is odd it is given by
\begin{equation}\label{curv-MB}
\begin{aligned}
\mathcal{R}^\rho(e_-,e_+) 		=\ & [\rho(e_+),\rho(e_-)]  
			=   \alpha^2+ \alpha\Gamma_+\otimes (\bar c-d)\,,  \\
\mathcal{R}^\rho(e_\mu, e_\nu) =\ & [\rho(e_\mu),\rho(e_\nu)] 
			= \alpha^2\Gamma_{\mu\nu} 
				 -\alpha\Gamma_+ \otimes \big \{s_{\bar c,d}(\Gamma_{[\mu}),\Gamma_{\nu]}\big\}\,,\\
\mathcal{R}^\rho(e_-,e_\mu) 
			=\ & [\rho(e_-),\rho(e_\mu)]- \Gamma([e_-,e_\mu])\,. \\
\end{aligned}
\end{equation}
In particular, for generic parameters the curvature is generic, too. The same holds in the even dimensional case. 
Although we will first emphasize on Clifford representations, i.e.\ flat connections, we will also consider non-flat connections later, see  Section \ref{restr}
\end{rem}

The examination of the $(W,W)$-brackets yields further conditions for the Clifford map to be a Clifford representation. We get
\begin{align}
&	\bar ca-ad =0  \,,\quad
	a\bar b=\bar b a = 0 \,,			
\label{23}\\
\intertext{from $[\RR_+,\RR_-]$,} 
&\begin{aligned}
	&\bar b\Gamma_{[\mu}\bar b\Gamma_{\nu]} = \Gamma_{[\mu}\bar b\Gamma_{\nu]}\bar b =0\,,\quad  
	&\Gamma_{[\mu}\bar b s_{\bar c,d}(\Gamma_{\nu]})-s_{\bar c,d}(\Gamma_{[\mu})\bar b\Gamma_{\nu]}=0  \,,
 \end{aligned}
\label{25}\\
\intertext{from $[V,V]$, and}
&	\Gamma_\mu(\bar b\bar c+d\bar b)-2\bar c\Gamma_\mu\bar b
				=(\bar b\bar c+d\bar b)\Gamma_\mu -2\bar b \Gamma_\mu d =0 \,,\quad
	\bar b\Gamma_\mu\bar b=0 \,, 
\label{26}\\
&	 q_{\bar c ,d}(v)+2s_{e\bar b,-\bar be}(v) = -Bv 
\label{27}
\end{align}
from $[\RR_-,V]$. 
Of course, (\ref{25}-1) is a consequence of (\ref{26}-2). 

Taking into account Proposition \ref{maps} and (\ref{23}-2) we get $a=b=0$ in the odd dimensional case. In the even dimensional case $\alpha^2=\beta^2$ for the two scalars such that $\bar b=-\alpha\Pi^\mp$ and $a=\alpha\Pi^\pm$ where $\Pi^\pm=\frac{1}{2}(\ei \pm \Gamma^*)$ denote the projections on the two half spinor subspaces of $S(V)$. We remark that (\ref{26}-2) is fulfilled in this situation. In the odd dimensional case and in the even dimensional case with $\alpha=0$ equations (\ref{23})-(\ref{26}-2) are satisfied such that we get the following result, which is in analogy to Theorem 2.3 in \cite{santi1}. 

\begin{thm}\label{alphazero}
Any simple irreducible Clifford representation of $M_B$ with $\rho(e_+)=0$ is given by 
\begin{equation}\label{rep-alpha0}
\begin{gathered}
\rho(v^*) 	=\frac{1}{2}\Gamma_+\,\hat\otimes\, Bv\,,\qquad
\rho(w)   	= -\frac{1}{2}\Gamma_+\,\hat\otimes\,s_{\bar c,d}(w)  \\
\rho(e_-) 	= \sigma_-\,\hat\otimes\, c +\sigma_+\,\hat\otimes\, d
			  +\Gamma_+\,\hat\otimes\, e 
\end{gathered}
\end{equation}
with $(\bar c,d)$ being a quadratic Clifford pair representing the symmetric map $-B$. In particular, $e\in \cl(V)$ is a free parameter of the representation.
\end{thm}

Now consider $\alpha\neq 0$, i.e.\ $a=\alpha\Pi^\pm$ and $b=-\alpha\Pi^\mp$. We restrict in the following to the upper sign and we denote the blocks of $\bar c$ and $d$ with respect to the half spinor splitting by $ c^+_-=\Pi^+ \bar c \Pi^-$, and so on. Equation (\ref{23}-1) yields $c^-_+=d^+_-=0$ and  $c^+_+=d^+_+=r$. The only surviving components of (\ref{26}-1) are
$
v^+_-(c^-_-+d^-_-)-2rv^+_-=0
$
and
$
(c^-_-+d^-_-)v^-_+-2v^-_+r=0
$. 
Therefore, $r$, $c^-_-$, and $d^-_-$ are scalars that obey $2r=c^-_-+d^-_-$, in particular, $r-d^-_-=-(r-c^-_-)$. This also solves the only surviving component of equation (\ref{25}-2), which in this situations reads
$
(d^-_- +  c^-_- - 2r)(v^+_-w^-_+ - w^+_-v^-_+)=0
$.
The mixed components of (\ref{27}) are
$
(r-d^-_-)^2v^+_--2\alpha v^+_-e^-_-=-(Bv)^+_-
$ 
and 
$
(d^-_--r)^2v^-_+-2\alpha e^-_-v^-_+=-(Bv)^-_+
$ 
and if we apply this to a basis of eigenvectors we see that the only symmetric map $B$ that can be realized has to be diagonal, $B=-2\lambda\ei$. In this situation the scalars obey 
$
(r-d^-_-)^2=(r-c^-_-)^2=2(\alpha e^-_-+\lambda)
$. 
The $\sfrac{+}{+}$-component of (\ref{27}) connects the off-diagonal blocks of $\bar c$, $d$, and $e$ via
$
(d^-_--r)(c^+_-v^-_+-v^+_-d^-_+)=2\alpha (e^+_-v^-_++v^+_-e^-_+)
$ 
or $(d^-_--r)s_{c^+_-,d^-_+}(v)=2\alpha s_{e^+_-,-e^-_+}(v)$ 
such that the only unconstrained component of $\rho(e_-)$ is $e^+_+$. 

We summarize the discussion above in the following proposition. 

\begin{thm}\label{alphanotzero}
A simple Clifford representation of $M_B$ with $\rho(e_+)\neq 0$ is only possible in even dimension and for $B=-2\lambda\ei$. It is given by 
\begin{equation}\label{rep-alphanot0}
\begin{aligned}
\rho(v^*) =\ & -\lambda\Gamma_+\,\hat\otimes\, v	\,,\qquad
\rho(e_+) = \alpha \Gamma_+\,\hat\otimes\, \Pi^\pm \,,	
\\
\rho(e_-) =\ & \rho_0\mathbbm{1}\, \hat \otimes\, \mathbbm{1} 
			  +(\alpha\Gamma_-+\beta\Gamma_+-\sqrt{2(\alpha\beta+\lambda)}\sigma_-
			  +\sqrt{2(\alpha\beta+\lambda)}\sigma_+)\,\hat\otimes\,\Pi^\mp 		\\
			& +\sigma_-\,\hat\otimes\,c^\pm_\mp
			  +\sigma_+\,\hat\otimes\,d^\mp_\pm
			  +\Gamma_+\,\hat\otimes\,(e^+_-+e^-_+ + e^\pm_\pm)
\\
\rho(v)  =\ &  \alpha \sigma_-\,\hat\otimes\, \Pi^\pm v +\alpha\sigma_+\,\hat\otimes\,\Pi^\mp v
			 +\sqrt{\frac{\alpha\beta+\lambda}{2}}\Gamma_+\,\hat\otimes\,v 
				+\frac{1}{2}\Gamma_+\,\hat\otimes\,s_{\bar c^\pm_\mp,d^\mp_\pm}(v) \,.
\end{aligned}
\end{equation}
The free parameters that describe the representation are the scalars $\alpha,\beta,\rho_0$ and the Clifford element $e^\pm_\pm$. The further contributions are related by 
$\sqrt{2}\,\alpha s_{e^\pm_\mp,-e^\mp_\pm}(v)=\sqrt{\alpha\beta+\lambda}\,s_{\bar c^\pm_\mp,d^\mp_\pm}(v)$ for all $v\in V$.
\end{thm}

\begin{rem}\label{rem-special}
A special choice of parameters in Proposition \ref{alphanotzero}, namely $\rho_0=0$, $\alpha=-\lambda$, $\beta=1$ as well as $c^\pm_\mp=d^\mp_\pm=e^+_-=e^-_+=e^\pm_\pm=0$ yields a representation that is similar to the one found in \cite{santi1}:
$\rho(v^*)=-\lambda\Gamma_+\,\hat\otimes\,v$, $\rho(e_+)=-\lambda\Gamma^+\,\hat\otimes\,\Pi^\pm$, $\rho(w)=-\lambda\sigma_-\,\hat\otimes\,\Pi^\pm w-\lambda\sigma_+\,\hat\otimes\,w\Pi^\mp$, and $\rho(e_-)=-\lambda\Gamma_-\,\hat\otimes\,\Pi^\mp+\Gamma_+\,\hat\otimes\,\Pi^\mp$.

This choice is the same as taking $c=d=0$ in (\ref{19})-(\ref{22}). Then (\ref{27}) reduces to $-2\alpha(e\Pi^\mp v+v\Pi^\pm e)=-Bv$. From Lemma \ref{root} we get that the solutions must obey $ e\Pi^\mp=\beta\Pi^\mp$ and $-\Pi^\mp e=-\beta\Pi^\mp$, i.e.\ $e=\beta\Pi^\mp +\Pi^\pm e'\Pi^\pm$ which is only possible for $B=-2\lambda\ei$ with $\lambda=-\alpha\beta$. Again, $e^\pm_\pm=\Pi^\pm e'\Pi^\pm$ is the same free parameter we found before.

Regardless of our assumption that led to Proposition \ref{alphanotzero} we may consider $\alpha=0$ and furthermore $\beta=\rho_0=0$ as well as $c^\pm_\mp=d^\mp_\pm=e^+_-=e^-_+=e^\pm_\pm=0$. This leads to a rather simple representation associated to $B=-2\lambda\ei$, also covered in Theorem \ref{alphazero}, namely
$\rho(v^*)=-\lambda\Gamma_+\,\hat\otimes\,v$, $\rho(e_+)=0$, $\rho(w)=\sqrt{\frac{\lambda}{2}}\,\Gamma_+\,\hat\otimes\, w $, $\rho(e_-)=-\sqrt{2\lambda}\,\sigma\,\hat\otimes\,\Pi^\mp$. The quadratic Clifford pair that defines this representation is pseudo-monomial in the sense of Remark \ref{def10}.
\end{rem}

\subsection{Non-simple Clifford representations of $M_B$}

For any non-simple Clifford representation the restriction of $\rho$ from $\mathfrak{g}\oplus\mathfrak{h}'$ to $\mathfrak{g}$ is a simple one. 
Therefore, the non-simple representations can be obtained by the simple ones by demanding the further compatibility with $\mathfrak{h}'$. As stated before, for a CW-space $M_B$ defined by the symmetric map $B$ the algebra $\mathfrak{h}'$ is a subalgebra of $\mathfrak{so}_B(V)$ which can be written as $\bigoplus_{\alpha} \mathfrak{so}(V_\alpha)$ where $V_\alpha$ are the distinct eigenspaces of $B$. Due to the nature of the Clifford representation $\rho$ we have $\rho(A)=A$ when we consider $A\in \cl_2(V)$ such that the further relations are
\begin{align}
[\rho(A),\rho(e_+)]&=0\,, \\
[\rho(A),\rho(e_-)]&=0\,, \\
[\rho(A),\rho(v^*)]&=\rho( (Av)^*)\,,\\
[\rho(A),\rho(w)]&=\rho(Aw)\,.
\end{align}
Regardless what comes afterwards, (\ref{19})-(\ref{22}) yield that the parameters that define the Clifford representation $\rho$ have to be invariant with respect to $\mathfrak{h}'$, i.e.\ $[A,a]=[A,b]=[A,c]=[A,d]=[A,e]=0$ for $A\in\mathfrak{h}'$. 
\begin{prop}\label{non-simple}
\begin{enumerate}
\item
The simple Clifford representations of Theo\-rem \ref{alphazero} extend to non-simple ones if and only if $c$, $d$ and $e$ are invariant with respect to $\mathfrak{h}'\subset\mathfrak{so}_B(V)$.  
\item
The simple Clifford representations of Theorem \ref{alphanotzero} extend to non-simple ones for all choices of sub algebras of $\mathfrak{so}(V)$ for the special choices of parameters as discussed in Remark \ref{rem-special}.
\end{enumerate}
\end{prop}
\begin{rem}\label{rem-soB-invariant}
The invariance of an element of the Clifford algebra $\cl(V)$ with respect to certain subalgebras of $\mathfrak{so}(V)$ is reflected in a special form of this element.
\begin{itemize}
\item
Let $V=V_1\oplus V_2$ be an orthogonal decomposition of $V$. Then $c\in\cl(V)=\cl(V_1)\,\hat\otimes\,\cl(V_2)$ is invariant with respect to $\mathfrak{so}(E_1)\subset\mathfrak{so}(V)$ if and only if $c\in \cl(V_2)+\Gamma_I\cl(V_2)$ with $\Gamma_I=\Gamma_{i_1}\cdots\Gamma_{i_{\dim(V_1)}}$ being proportional to the volume element of $V_1$ in $\cl(V_1)\subset\cl(V)$. 
\item
Symmetrizing this example and extending to more than two factors yields the following generalization.  
Consider $\mathfrak{h}'=\bigoplus_\alpha\mathfrak{so}(V_\alpha)$ with $\bigoplus_\alpha V_\alpha =V$ being an orthogonal decomposition. Denote by $\Gamma_{I_\alpha}$ some multiple of the volume form of $V_\alpha$ in $\cl(V)$. Then $c\in\cl(V)={\widehat\bigotimes}_\alpha \cl(V_\alpha)$ is invariant with respect to $\mathfrak{h}'$ if and only if the homogeneous components of $c$ are proportional to products of volume forms of the different factors. For example, the elements that yield quadratic Clifford pairs as considered in Propositions \ref{monomial-th}, \ref{prop2}, \ref{hom}, and \ref{hom-p} are of this type. 
\end{itemize}
\end{rem}

\subsection{Restrictions of the Clifford module}\label{restr}
\subsubsection{Canonical restrictions}

When we talk about Clifford maps and Clifford representations $\rho$ we may ask if we are able to consider subspaces $S'$ of the Clifford module $S$ such that 
the restriction to $S'$ 
\begin{itemize} 
\item obeys $\rho(x)S'\subset S'$ for all $x\in V^*\oplus W$ or 
\item $\rho$ is a representation of $V^*\oplus W$ on $S'$
\end{itemize}
although $S'$ is not a Clifford module. Of course, for a Clifford represenation $\rho$ both conditions coincide.
As we mentioned in Remark \ref{rem:restrKill} Clifford maps that yield represenations on the restrictions are strongly related to connections for which the space of parallel sections of the associated spinor bundle doesn't have maximal dimension. We will call a  restriction of $S$ that is compatible with the action of $\mathfrak{so}(V)\subset\mathfrak{so}(W)$ a {\em canonical restriction}. In particular, such restrictions are compatible with any non-simpli\-ci\-ty factor $\mathfrak{h}'\subset\mathfrak{so}(V)$.

If $\dim(V)$ is odd, there are two subspaces defined by the chiral structure on the two-dimensional factor of $W$. If we use the notations of Section \ref{solvable}, i.e.\ $S_2^{\pm}={\rm ker}(\Gamma_\mp)={\rm ker}(\sigma_{\mp})={\rm Eig}(\sigma,\pm)$, these subspaces are given by $S_2^-\otimes S(V)=S(V)\oplus\{0\}$ and $S^+_2\otimes S(V)=\{0\}\oplus S(V)$.

If $\dim(V)$ is even we have a decomposition into four canonical subspaces because each of the former is further decomposed with respect to the chiral structure on the second factor. The possible spaces are  
$S_2^+\otimes S^+(V)=\{0\}\oplus S^+(V)$, 
$S_2^-\otimes S^-(V)=S^-(V)\oplus\{0\}$, 
$S_2^+\otimes S^-(V)=\{0\}\oplus S^-(V)$, and 
$S_2^-\otimes S^+(V)=S^+(V)\oplus\{0\}$. In particular, the first two sum up to 
$S^+(W)=S^-(V)\oplus S^+(V)$ and the remaining ones to 
$S^-(W)=S^+(V)\oplus S^-(V)$.

We will discuss these two cases separately with regard to Section \ref{sec:simple-irred}.

{\em $\dim(V)$ odd:}\ \ 
In this case we consider (\ref{rep-alpha0}) from Theorem \ref{alphazero}. The reduction to $S_2^-\otimes S(V)$ immediately implies $S'$ to be trivial as the action of $V^*$ shows. The reduction to $S^+_2\otimes S(V)$ yields a trivial action of $V^*$ and $V$, and $\rho(e_-)(\xi_1)=\bar c\xi_1$. In particular any Clifford map yields a representation on such reduction. 

{\em $\dim(V)$ even:}\ \ 
Firstly, we consider Clifford maps of the form (\ref{rep-alpha0}) and go through the different cases of restrictions. For this we write it in terms of the canonical decomposition: 
\begin{equation}
\begin{gathered}
\rho(v^*)\vec\xi =\frac{1}{\sqrt{2}}
		\begin{pmatrix}
				Bv\xi_2^- \\ Bv\xi_2^+ \\ 0 \\ 0
		\end{pmatrix},\quad 
\rho(w)\vec\xi  = -\frac{1}{\sqrt{2}}
		\begin{pmatrix}
				s_{\bar c_-^+,d^-_+}(w)\xi_2^+ +s_{\bar c_+^+,d_-^-}(w)\xi_2^-\\
				s_{\bar c_-^-,d_+^+}(w)\xi_2^+ +s_{\bar c_+^-,d_-^+}(w)\xi_2^-\\0\\0 
		\end{pmatrix},\\ 
\rho(e_-) \vec\xi = 
		\begin{pmatrix}
			\bar c^+_+\xi_1^+ + \bar c^+_-\xi_1^- + \sqrt{2}e^+_+\xi_2^+ + \sqrt{2}e_-^+\xi_2^-\\
			\bar c^-_+\xi_1^+ + \bar c^-_-\xi_1^- + \sqrt{2}e^-_+\xi_2^+ + \sqrt{2}e_-^-\xi_2^-\\
			d_+^+\xi_2^+ + d_-^+\xi_2^-\\
			d^-_+\xi_2^+ + d_-^-\xi_2^-\\
		\end{pmatrix}. 
\end{gathered}			 
\end{equation}
We restrict to the upper sign and discuss the different cases.

{$\bullet$ [$S'=S(V)\oplus S^\mp(V)$, i.e.\ $\xi_2^\pm=0$] }\ 
We consider the upper sign and see that $\rho$ restricts to a representation if and only if $d^+_-=0$. In this case we have
\begin{equation*}
\begin{gathered}
\rho(v^*){\begin{pmatrix}\xi_1^+\\\xi_1^-\\\xi_2^-\end{pmatrix}}
	 =\frac{1}{\sqrt{2}}
		\begin{pmatrix}
				Bv\xi_2^- \\ 0 \\ 0
		\end{pmatrix},\quad 
\rho(w){\renewcommand{\arraystretch}{1.3}\begin{pmatrix}\xi_1^+\\\xi_1^-\\\xi_2^-\end{pmatrix}}
	 = -\frac{1}{\sqrt{2}}
		\begin{pmatrix}
				 s_{\bar c_+^+,d_-^-}(w)\xi_2^-\\
				 c_+^-w \xi_2^-\\0 
		\end{pmatrix}, 	\\		
\rho(e_-) {\renewcommand{\arraystretch}{1.3}\begin{pmatrix}\xi_1^+\\\xi_1^-\\\xi_2^-\end{pmatrix}} = 
		\begin{pmatrix}
			\bar c^+_+\xi_1^+ + \bar c^+_-\xi_1^- + \sqrt{2}e_-^+\xi_2^-\\
			\bar c^-_+\xi_1^+ + \bar c^-_-\xi_1^- + \sqrt{2}e_-^-\xi_2^-\\
			d_-^-\xi_2^-\\
		\end{pmatrix} .
\end{gathered}			 
\end{equation*}
In the case of the lower sign the restriction of $\rho$ yields a representation only if $d^-_+=0$ and is treated similarly with an adapted change of notation.

\begin{rem}\label{rem:restr}
In this situation let us consider the quadratic Clifford pairs from Section~\ref{sec:qcp}.

For a  monomial quadratic Clifford pair the calculation above only holds for $\hat I$ even. In this case $c_+^-=c_-^+=0$ as well, such that only the first component of $\rho(v^*)\vec\xi$ and $\rho(w)\vec\xi$ survives.
The same is true for $(c,d)$ pseudo-monomial and even, i.e.\ $c=\pm d=(\alpha_+\Pi^++\alpha_-\Pi^-)\Gamma_I$ or even generalized monomial with $c=d=\sum_\alpha c_\alpha\Gamma_{I_\alpha}$. Such pairs yield a representation on the reduction only if the pair is associated to $B$, i.e.\ $q_{c,d}(e_\mu)\xi= B(e_\mu)\xi$ for all $\xi\in S'$.

The situation is different for a pseudo-monomial quadratic Clifford pair with $\hat I$ odd, e.g.\ $c=\alpha(\gamma+\Gamma^*)\Gamma_I,d=\beta(\gamma-\Gamma^*)\Gamma_I$. We have $d^+_-=\frac{\alpha(1-\gamma)}{2}\Pi^+\Gamma_I$, such that $\gamma=1$ is fixed. Therefore, we are left with $c=c_-^+=\alpha\Pi^+\Gamma_I$ and $d=d_+^-=\beta\Pi^-\Gamma_I$ and $B=0$ is mandatory if we look for representations.
\end{rem}

{$\bullet$ [$S'=S^\mp(W)=S^\pm(V)\oplus S^\mp(V) $, i.e.\ $\xi_1^\mp=\xi_2^\pm=0$] }\ 
The restriction to $\xi^\mp_1=0$ immediately yields a further restriction to $\xi_2^\pm=0$ due to the action of $V^*$.  Therefore, we are left with a special subcase of the first reduction: $S'=S^\mp(W)$. In particular, the consequences of Remark \ref{rem:restr} still hold.

{$\bullet$ [$S'=S(V)\oplus \{0\}$, i.e.\ $\xi_2^+=\xi^-_2=0$] }\  
This case is treated analogously to the odd dimensional case, i.e.\ the action of $\RR_-$ alone survives and is given by $\rho(e_-)(\xi_1)=\bar c\xi_1$.  

{$\bullet$ [remaining cases] }\ 
All further canonical reductions of $S$ yield further subcases of the discussed ones or forces $S'$ to be trivial.

Next, we turn to (\ref{rep-alphanot0})  of Theorem \ref{alphanotzero} and do as before. Here we restrict to Clifford representations. We recall
\begin{equation}\tag{\ref{rep-alphanot0}}
\begin{gathered}
\rho(v^*)\vec\xi =-\sqrt{2}\lambda
		{\renewcommand{\arraystretch}{1.3}\begin{pmatrix}
				v\xi_2^- \\ v\xi_2^+ \\ 0 \\ 0
		\end{pmatrix}},\quad  
\rho(e_+)\vec\xi = \alpha 
		{\renewcommand{\arraystretch}{1.3}\begin{pmatrix}
				\xi_2^+ \\ 0\\ 0 \\ 0
		\end{pmatrix}},\\
\rho(w)\vec\xi   =  
		{\renewcommand{\arraystretch}{1.3}\begin{pmatrix}
				 -\alpha w\xi_1^- + \sqrt{\alpha\beta+\lambda} w\xi_2^- 
				 +\frac{1}{\sqrt{2}} s_{\bar c_-^+,d_+^-}(w)_+^+\xi_2^+  \\
				 \sqrt{\alpha\beta+\lambda} w\xi_2^+ \\
				 0\\
				 \alpha w\xi_2^+ 
		\end{pmatrix}}, \\		
\rho(e_-)\vec\xi =  
		{\renewcommand{\arraystretch}{1.3}\begin{pmatrix}
		  	\rho_0\xi_1^+	+ \bar c_-^+\xi_1^- +\sqrt{2} e_+^+\xi_2^+ +\sqrt{2} e_-^+\xi_2^- \\
		  	\rho_0\xi_1^-	+\sqrt{2}\beta\xi_2^- -\sqrt{2(\alpha\beta+\lambda)}\xi_1^-+\sqrt{2} e_+^-\xi_2^+ \\
		  	\rho_0\xi_2^+ \\
		  	\rho_0\xi_2^- -\sqrt{2}\alpha\xi_1^- +\sqrt{2(\alpha\beta+\lambda)}\xi_2^- +d_+^-\xi_2^+
		\end{pmatrix}}.	
\end{gathered}	
\end{equation}
and again go through the different cases.

{$\bullet$ [$S'=S(V)\oplus S^-(V)$, i.e.\ $\xi_2^+=0$] } 
In contrast to the preceding discussion the symmetry in $\xi_2^+$ and $\xi_2^-$ is broken, because the projections on the two factors enter into the representation. For $\xi^+_2=0$ we see that the action of $\RR_+$ is trivial. We are left with
\begin{equation*}
\begin{gathered}
\rho(v^*){\renewcommand{\arraystretch}{1.3}\begin{pmatrix}\xi_1^+\\\xi_1^-\\\xi_2^-\end{pmatrix}}
	 =-\sqrt{2}\lambda
		{\renewcommand{\arraystretch}{1.3}\begin{pmatrix}
				v\xi_2^- \\ 0  \\ 0
		\end{pmatrix}}\,, 
\ 
\rho(w) {\renewcommand{\arraystretch}{1.3}\begin{pmatrix}\xi_1^+\\\xi_1^-\\\xi_2^-\end{pmatrix}}
	=  
		{\renewcommand{\arraystretch}{1.3}\begin{pmatrix}
				 -\alpha w\xi_1^- +\sqrt{\alpha\beta+\lambda} w\xi_2^-   \\
				 0 \\
				 0 
		\end{pmatrix}}\,,\\		
\rho(e_-){\renewcommand{\arraystretch}{1.3}\begin{pmatrix}\xi_1^+\\\xi_1^-\\\xi_2^-\end{pmatrix}}
	 =  
		{\renewcommand{\arraystretch}{1.3}\begin{pmatrix}
		  		\rho_0\xi_1^+	+ c_-^+\xi_1^- +\sqrt{2} e_-^+\xi_2^- \\
		  		\rho_0\xi_1^-	+\sqrt{2}\beta\xi_2^- -\sqrt{2(\alpha\beta+\lambda)}\xi_1^-\\
		  		\rho_0\xi_2^- -\sqrt{2}\alpha\xi_1^- +\sqrt{2(\alpha\beta+\lambda)}\xi_2^-
		\end{pmatrix}}\,.	\\
\end{gathered}	
\end{equation*}
This is example is not covered by Theorem \ref{alphazero} although the action of $\RR_+$ is trivial because $b$ in (\ref{22}) does not vanish. This is of course no contradiction but a consequence of $S'$ not being a Clifford module. Therefore, identities (\ref{23})-(\ref{27}) only have to be satisfied after being applied to the subset $S'\subset S$. In fact, the above is a solution of this relaxed condition.

{$\bullet$ [$S'=S^-(W)=S^+(V)\oplus S^-(V)$, i.e.\ $\xi_2^+=\xi_1^-=0$] }\ 
For $\xi_1^-=0$ the action of $V^*$  immediately yields $\xi_2^+=0$. This is a possible reduction for $\beta=0$ as can be seen from the case before. 

{$\bullet$ [$S'=S^+(V)\oplus\{0\}$, i.e.\ $\xi_2^+=\xi_2^-=\xi_1^-=0$] }\ 
Starting with $\xi_2^-=0$ and assuming $\alpha\neq 0$ the actions of $V$ and $\RR_-$ immediately yield a further reduction to $\xi_2^+=\xi_1^-=0$. But then, again from the first subcase, $V$, $V^*$, and $\RR_+$ act trivially and $\rho(e_-)(\xi_1^+)=\rho_0 \xi_1^+$. 

{$\bullet$ [remaining cases] }\  
Starting with $\xi_1^+=0$ the action of $V^*$ yields $\xi_2^-=0$ such that $S'=0$ in this case.

\subsubsection{A note on non-canonical restrictions}

Up to now we discussed canonical restrictions but there are interesting non canonical restrictions, too. A basic example for a Clifford representation cf.\ Theorem \ref{alphazero} is obtained as follows. Consider a monomial quadratic Clifford pair $(c,d)$ with $c=\alpha\Gamma_I$, $d=\beta\Gamma_I$. Then $\Gamma_I^2=\sigma_I\mathbbm{1}$ such that $X_\pm^I:= \frac{1}{2}(\mathbbm{1}\pm \sqrt{\sigma_I} \,\Gamma_I)$ are the two projections on the $\pm1$-eigenspaces of $\sqrt{\sigma_I}\,\Gamma_I$. Then 
\begin{equation}
S':=S(V)\oplus X_+^IS(V)\subset S
\end{equation} 
is a possible restriction such that $\rho$ remains a representation. This restriction is not compatible with the action of $\mathfrak{so}(V)$, thus not canonical, but compatible with the action of 
$\mathfrak{so}(\hat I)\oplus\mathfrak{so}(n-\hat I)
	=\mathfrak{so}_B(V)\subset\mathfrak{so}(V)$. 

This example can be extended to pairs $(c,d)$ according to the approach in Proposition \ref{moregen1} that do not yield quadratic Clifford pairs. Therefore, we consider a pair $c=(\alpha\Gamma_I+\beta\Gamma_J)\Gamma_K$, $d=(\alpha'\Gamma_I+\beta'\Gamma_J)\Gamma_K$. This pair obeys 
\begin{equation}\label{1}
q_{c,d}(e_\mu)=\alpha_+^\mu\Gamma_\mu X_+^{IJ} + \alpha_-^\mu\Gamma_\mu X_-^{IJ}
\end{equation}
where $X^{IJ}_\pm=\frac{1}{2}(\mathbbm{1}\pm \imath_{IJ}\Gamma_{IJ})$. Here $\imath_{IJ}^2=\sigma_{IJ}=(-1)^{\hat I\hat J}\sigma_I\sigma_J$ such that $X_\pm^{IJ}$ are the projections on the $\pm1$-eigenspaces of $\imath_{IJ}\Gamma_{IJ}$. The numbers  $\alpha_\pm^\mu$ run through the eight values $\alpha\pm\beta\pm\alpha'\pm\beta'$.

We consider the restriction
\begin{equation}
S'=\ker(\Gamma_+\otimes X_-^{IJ})S=S(V)\oplus X^{IJ}_+S(V)\,.
\end{equation}
Then $\rho$ restricted to $S'$ is a representation if and only if it is compatible with the action of $\RR_-$, i.e.\ $d X^{IJ}_+ \xi\in X^{IJ}_+S(V)$ or $X^{IJ}_-dX^{IJ}_+=0$ for $\xi\in S(V)$. A careful calculation yields 
\begin{equation}
4 X^{IJ}_\mp d X_\pm^{IJ} = \big(1-(-1)^{\hat I\hat J+(\hat I+\hat J)\hat K}\big)
		\big((\alpha'\mp\beta'\imath_{IJ}\sigma_J)\Gamma_{IK}
			+(\beta'\mp\alpha'\imath_{IJ}\sigma_I(-1)^{\hat I\hat J})\Gamma_{JK}\big)
\end{equation}
which vanishes if and only if 
\begin{equation}
\alpha'=\pm\imath_{IJ}\sigma_I\beta'
	\quad\text{ or } \quad 
(-1)^{\hat I\hat J+(\hat I+\hat J)\hat K}=1\,.
\end{equation}
In the first case we get the more special result $d X_\pm=0$ such that the construction of the Clifford map is independent of $d$ and we may assume $d=0$. In this case we have 
\begin{equation}
q_{c,d}(e_\mu)=c^2\Gamma_\mu
		=(\sigma_{IK}\alpha^2+\sigma_{JK}\beta^2)\Gamma_\mu 
		=:-\lambda\Gamma_\mu 
\end{equation}
such that for $\vec{\xi}=(\xi_1,\xi_2)\in S'$ the Clifford map $\rho$ is given by 
\begin{equation}
\begin{aligned}
 \rho(e_-)\vec\xi    &=(c\xi_1,0)^t\,,
&\rho(e_\mu^*)\vec\xi&=\frac{\lambda}{\sqrt{2}}(\Gamma_\mu\xi_2,0)^t\,,
&\rho(e_\mu)\vec\xi  &=-\frac{1}{\sqrt{2}}(c\Gamma_\mu\xi_2,0)^t\,.
\end{aligned}
\end{equation}
This is in fact a representation on $S'$ for $B=\lambda\mathbbm{1}$.

In the generic situation the pair $(c,d)$ and the associated Clifford map $\rho$ are not compatible with the action of $\mathfrak{so}(V)$ but with the action of 
$\mathfrak{so}(\hat I)\oplus\mathfrak{so}(\hat J)\oplus\mathfrak{so}(\hat K) \oplus\mathfrak{so}(n-\hat I-\hat J-\hat K)\subset\mathfrak{so}(V)$. $\rho$ is a representation on $S'$ if $B$ is defined by (\ref{1}), i.e.\ $B(e_\mu)X^{IJ}_+= q_{c,d}(e_\mu) X^{IJ}_+$.

\section{Distinguished quadratic Clifford pairs}

\subsection{A compatibility condition}

We show in this section, that the quadratic Clifford pairs we constructed in Section \ref{sec:qcp}, Propositions \ref{monomial-th}, \ref{prop2}, \ref{hom}, and \ref{hom-p} yield a complete list in a particular way. 

Let $(c,d)\in\cl(V)\times\cl(V)$ be a quadratic Clifford pair associated to the symmetric map $B$, i.e.\ $B(v)=q_{c,d}(v) $ for $v\in V$. Then we know from Remark \ref{rem-soB-invariant} that $(c,d)$ is invariant with respect to $\mathfrak{so}_B(V)$ if it is from Section \ref{sec:qcp}. 
We associate to $(c,d)$ an element $\Omega_{c,d}\in V^*\otimes V^*\otimes\cl(V)$ defined by 
\begin{equation}\label{Omega}
\Omega_{c,d}(v,w):= s_{d,c}(v) w + w s_{c,d}(v)
\end{equation}
for all $v,w\in V$. 
For an orthonormal basis $\{e_\mu\}$ of $V$ (\ref{Omega}) reads as
\begin{equation}\label{Omega-ij}
\Omega_{c,d}(e_\mu,e_\nu) =
2\Gamma_{[\mu}c\Gamma_{\nu]} - \{\Gamma_{\mu\nu},d\}\,.
\end{equation}
Consider homogeneous elements $c=c_I\Gamma_I,d=d_I\Gamma_I\in\cl(V)$ and let $\theta^\mu_I=-1$ if $\mu\in I$ and $\theta^\mu_I=1$ if $\mu\not\in I$. Then
\begin{equation}\label{64}
\begin{aligned}
\Omega_{c,d}(e_\mu,e_\nu)
&= \big((\theta^\mu_I+\theta^\nu_I)(-1)^{\hat I}c_I
   -(\theta^\mu_I\theta^\nu_I+1)d_I\big)\Gamma_I\Gamma_{\mu\nu} \\
&=\begin{cases}
 -2((-1)^{\hat I}c_I+d_I)\Gamma_I\Gamma_{\mu\nu}  & \text{ if } \mu,\nu\in I\,,\\
 2((-1)^{\hat I}c_I-d_I)\Gamma_I\Gamma_{\mu\nu} & \text{ if } \mu,\nu\not\in I\,,\\
 0 & \text{ else}\,.
\end{cases}
\end{aligned}
\end{equation}
To simplify the computation let $\lambda_1,\ldots,\lambda_r$ be the distinct eigenvalues of $B$ and consider the decomposition of $V=\bigoplus_{\alpha=1}^r V_\alpha$ into the respective eigenspaces. The dimension of $V_\alpha$ is denoted by $\hat\alpha$ and let $\{e_{\alpha,\mu}\}_{\mu=1,\ldots,\hat\alpha;\alpha=1,\ldots,r}$ be an adapted orthonormal basis. Then $\Gamma_{I_\alpha}=\Gamma_{\alpha,1}\cdots\Gamma_{\alpha,\hat\alpha}$ with $\hat I_\alpha=\hat\alpha$ represents a multiple of the volume form of $V_\alpha$ in $\cl(V)$. 
A consequence of (\ref{64}) is the following proposition.
\begin{prop}\label{c+d}
Let $c,d\in\cl(V)$ be invariant with respect to $\mathfrak{so}_B(V)$. Then $\Omega_{c,d}\in \mathfrak{so}_B(V)\otimes\cl(V)$ if and only if{\,}\footnote{We consider the Clifford elements to be expanded in the basis that is adapted to the eigenspace decomposition of $V$ with respect to $B$.}
\begin{equation}\label{kl2}
\begin{aligned}
c&= c_0 +\hat c_0\Gamma^*
\\
d&= d_0 +\hat d_0\Gamma^* 
\end{aligned}
\end{equation}
if $B=\lambda\mathbbm{1}$, 
\begin{equation}\label{gl2}
\begin{aligned}
c&=(c_0  + \hat c_0\Gamma^*) + (c_1+ \hat c_1\Gamma^*)\Gamma_{I_1} \\
d&=(c_0  -(-1)^{{\rm dim}(V)}\hat c_0\Gamma^*) + (d_1 + \hat d_1\Gamma^*)\Gamma_{I_1}
\end{aligned}
\end{equation}
if the number of different eigenvalues of $B$ is two, and
\begin{equation}
\begin{aligned}\label{gr2}  
c&=(c_0 +\hat c_0\Gamma^*)+\sum_\alpha (c_\alpha+\hat c_\alpha\Gamma^*)\Gamma_{I_\alpha} 
\\
d&=(c_0 -(-1)^{{\rm dim}(V)}\hat c_0\Gamma^*) 
	+\sum_\alpha (-1)^{\hat I_\alpha} (c_\alpha-(-1)^{{\rm dim}(V)}\hat c_\alpha\Gamma^*)\Gamma_{I_\alpha} 
\end{aligned}
\end{equation}
if the number of different eigenvalues of $B$ exceeds two.
If the dimension of $V$ is odd we may choose $\hat c_\alpha=\hat d_\alpha=0$ because $\Gamma^*=\pm\ei$. 
\end{prop}

\begin{proof}
We consider an orthonormal basis $\{e_\mu\}$ adapted to the eigenspace decomposition with respect to $B$ as introduced above and write $\Omega_{\mu\nu}=\Omega_{c,d}(e_\mu,e_\nu)$. 
Then $\Omega_{c,d}\in\mathfrak{so}_B(V)\otimes\cl(V)$ if and only if $\Omega_{\mu\nu}=0$ for $e_\mu,e_\nu$ belonging to different eigenspaces.
Consider $c,d$ to be invariant with respect to $\mathfrak{so}_B(V)$ such that 
\[
c=\sum_{0\leq t\leq r}\sum_{\alpha_1<\cdots<\alpha_t}
			c_{\hat\alpha_1\ldots \hat\alpha_t}\Gamma_{I_{\alpha_1}}\cdots\Gamma_{I_{\alpha_t}}
\]
and $d$ analogously, see Proposition \ref{non-simple} and  Remark \ref{rem-soB-invariant}. 
Because of (\ref{Omega-ij}) we may apply (\ref{64}) to each homogeneous component of $c$ and $d$. 
First, we notice that for the term with $t=0$, i.e.\ the scalar part, only the second case of (\ref{64}) is present such that $c_\emptyset=d_\emptyset$. 
Next consider a term with $1<t<r-1$, e.g.\  $\Gamma_I=\Gamma_{I_{\alpha_1}}\Gamma_{I_{\alpha_2}}\cdots\Gamma_{I_{\alpha_t}}$. Take $e_\mu\in E_{\alpha_1},e_\nu\in E_{\alpha_2}$, then the first case of (\ref{64}) implies $(-1)^{\hat I}c_I+d_I=0$. Similarly, consider $e_\mu\in E_{\beta_1},e_\nu\in E_{\beta_2}$ with $\beta_1\neq\beta_2$ and $(I_{\beta_1}\cup I_{\beta_2})\cap I=\emptyset$ then the second case in (\ref{64}) implies $(-1)^{I}c_I+d_I=0$ such that $c_I=d_I=0$. 
Now consider $t=1$. If there exists more than two eigenspaces the second case in (\ref{64}) implies $(-1)^{\hat I}c_I-d_I=0$. If there exists exactly two eigenspaces $E_1$ and $E_2$ there is no such restriction because $\mu,\nu\not\in I_1$ is equivalent to $\mu,\nu\in I_2$. 
Next consider $t=r$. Then $\Gamma_I$ represents the volume form of $V$ and only the first case in (\ref{64}) is present such that $(-1)^{\hat I}c_I+d_I=0$. Last but not least consider $t=r-1$ for $r>3$. Then case two in (\ref{64}) yields no restriction but the first requires $(-1)^{\hat I}c_I+d_I=0$. 
In the case where $B\sim \ei$ there are no restrictions on $c$ and $d$.
The results are collected in (\ref{kl2})-(\ref{gr2}).
\end{proof}
\begin{rem}\label{rem:simple}
A similar discussion as before yields that $\Omega$ vanishes identically if and only if $c=\alpha+v, d=\alpha-v$ in odd dimensions or $c=(\alpha+v)+(\beta+w)\Gamma^*,d=(\alpha-v)-(\beta-w)\Gamma^*$ in even dimension. Here $\alpha,\beta$ are scalars and $v, w$ are orthogonal vectors.
\end{rem}
If we restrict Proposition \ref{c+d} to quadratic Clifford pairs we get the following result.
\begin{thm}\label{thm-c+d}
Let $(c,d)$ be a quadratic Clifford pair associated to the symmetric map $B$. Then $\Omega_{c,d}\in \mathfrak{so}_B(V)\otimes\cl(V)$  if and only if $(c,d)$ is
\begin{itemize}
\item 
a pair of scalars or a pair of pseudo scalars if $B\sim\mathbbm{1}$, 
\item
a monomial or a pseudo-monomial quadratic Clifford pair according to Proposition \ref{monomial-th} if $B$ admits two different eigenvalues, and
\item
a generalized monomial quadratic Clifford pair according to Propositions \ref{hom} and \ref{hom-p} if $B$ admits more than two different eigenvalues. 
\end{itemize}
In fact, the multiplicity of the eigenvalues is even for all up to at most two exceptions. 
\end{thm}

\begin{proof}
Proposition \ref{c+d} yields a necessary form for $c$ and $d$ and we examine the quadratic Clifford condition.

If $B=-\lambda\mathbbm{1}$ we may assume $d_0=0$. If furthermore $c_0\neq 0$ in (\ref{kl2}) we get a quadratic Clifford pair if and only if $\hat c_0=-\hat d_0$. In this Situation $q_{c,d}\big|_V$ is independent of the pseudo scalar parts such that we may choose $(c,d)=(\sqrt{\lambda},0)$. If we instead assume $c_0=0$ we get a monomial solution defined by two pseudo scalars. 
If $B$ has more than one eigenvalue we see that $q_{c,d}\big|_{V}$ is independent of the scalar as well as the pseudo scalar part such that we my assume $c_0=\hat c_0=0$. 
If $B$ admits two different eigenvalues the possible values are
\begin{equation*}
c=(c_1+ c_2\Gamma^*)\Gamma_{I}\,,\quad d=(d_1 + d_2\Gamma^*)\Gamma_{I}\,,
\end{equation*} 
and they yield a quadratic Clifford pair if and only if the pair is (pseudo)-monomial, see Proposition \ref{monomial-th}.
In the case that $B$ has more than two different eigenvalues $c,d$ are given by 
\begin{equation*}
c=\sum_\alpha (c_\alpha + \hat c_\alpha\Gamma^*)\Gamma_{I_\alpha}\,,\quad
d=\sum_\alpha (-1)^{\hat I_\alpha} (c_\alpha -\hat c_\alpha\Gamma^*)\Gamma_{I_\alpha} \,,
\end{equation*}
and they form a quadratic Clifford pair if only if they are of the form described in Propositions \ref{hom} and \ref{hom-p}.
\end{proof}
We close the discussion with two non trivial algebraic properties that are proved by a straight forward calculation.
\begin{rem}
All quadratic Clifford pairs $(c,d)$ according to Theorem \ref{thm-c+d} obey
\begin{align*}
&d s_{d,c}(e_\nu)\Gamma_\mu+d\Gamma_\mu s_{c,d}(e_\nu)-s_{d,c}(e_\nu)\Gamma_\mu d-\Gamma_\mu s_{c,d}(e_\nu)d=0\\\intertext{and}
&s_{d,c}(e_\nu)s_{c,d}(e_\mu)+s_{d,c}(e_\mu)s_{c,d}(e_\nu) = B_{\mu\nu}\,.
\end{align*}
\end{rem}

\subsection{Concluding remark}

As we mentioned in the introduction the spaces and connections discussed in this text play an important role in the treatment of backgrounds of supergravity theories. For this we have to consider enlargements of the Lie algebra of infinitesimal isometries of a given space to a super Lie algebra. The odd part of it is characterized by spinors that are parallel with respect to a given connection, see \cite{Cortes, KlinkerSSKS, KlinkerTor, KlinkerDeform}, for example, in addition to the references from the introduction.
Starting from the text at hand, the next natural step is the classification of super algebras that extend the isometries of CW-spaces. Work on this is in progress and the tensor $\Omega$ apparently plays a key role in this study.


\begin{appendix}
\section{Some technical proofs}
\markright{\textsl{\textsc{Some Technical Proofs}}}

\subsection{Proofs from Section \ref{sec:qcp}}\label{App1}

We take over the notations of Section \ref{sec:qcp} and consider $c,d\in\cl(V)$ to be the sum of two monomials and of the same type. More precisely, we consider $c=(\alpha\Gamma_I+\beta\Gamma_J)\Gamma_K\neq0$ and $d=(\alpha'\Gamma_I+\beta'\Gamma_J)\Gamma_K\neq0$ with $I\cap J=I\cap K=J\cap K=\emptyset$ and set $\Gamma_I^2=\sigma_I\mathbbm{1}$ and the same for $J$ and $K$. Of course, the results will be symmetric in $I$ and $J$. We use the notation $\theta^\mu_I=1$ if $\mu\not\in I$ and $\theta^\mu_I=-1$ if $\mu\in I$. The follwoing calculations yield the proof of Proposition \ref{moregen1}. We have
\begin{align*}
s_{c,d}(\Gamma_\mu)
=\ &	\alpha\Gamma_I\Gamma_K\Gamma_\mu + \beta\Gamma_J\Gamma_K\Gamma_\mu - \alpha'\Gamma_\mu\Gamma_I\Gamma_K - \beta'\Gamma_\mu\Gamma_J\Gamma_K \\
=\ &	(\alpha-(-1)^{\hat K+\hat I}\theta^\mu_I\theta^\mu_K \alpha')\Gamma_I\Gamma_K\Gamma_\mu 
		+(\beta-(-1)^{\hat K+\hat J}\theta^\mu_J\theta^\mu_K \beta')\Gamma_J\Gamma_K\Gamma_\mu 
\end{align*}
such that 
\begin{align}
\begin{aligned} 
q_{e,f}(e_\mu)
=\ &\Big( 
		 (-1)^{\hat I\hat K}\sigma_I\sigma_K
		 (\alpha-(-1)^{\hat K+\hat I}\theta^\mu_I\theta^\mu_K \alpha')^2\\
&\quad	+(-1)^{\hat J\hat K}\sigma_J\sigma_K
		 (\beta-(-1)^{\hat K+\hat J}\theta^\mu_J\theta^\mu_K \beta')^2
	 \Big)\Gamma_\mu  \\		
& +\sigma_K\Big(  \alpha\beta\big((-1)^{\hat J\hat K} +(-1)^{\hat I\hat K+\hat I\hat J} \big)\\
&\quad	 +\alpha'\beta'\big((-1)^{\hat J\hat K} +(-1)^{\hat I\hat K+\hat I\hat J} \big)
					(-1)^{\hat J+\hat I} \theta^\mu_I\theta^\mu_J\\
&\quad		 +2 \alpha\beta' (-1)^{(\hat J+1)(\hat K+1)}\theta^\mu_J\theta^\mu_K   \\
&\quad		 +2 \alpha'\beta (-1)^{(\hat I+1)(\hat K+1)+\hat I\hat J}\theta^\mu_I\theta^\mu_K
	\Big)\Gamma_I\Gamma_J\Gamma_\mu\,. \label{formel}
\end{aligned}
\end{align}
$q_{c,d}$ restricted to $V$ is a self map if and only if the summand proportional to $\Gamma_I\Gamma_J\Gamma_\mu$ vanishes (at least, when it is present). This yields a homogeneous linear system for the four coefficients $\alpha\beta$, $\alpha'\beta'$, $\alpha\beta'$, and $\alpha'\beta$ with matrix\footnote{The rows are related to $\theta^\mu_J=-1$, $\theta^\mu_I=-1$, $\theta^\mu_I=\theta^\mu_J=\theta^\mu_K=1$, and $\theta^\mu_K=-1$. The columns are related to  $\alpha\beta$, $\alpha'\beta'$, $\alpha\beta'$, and $\alpha'\beta$.} 
\begin{multline}\small
\left({\renewcommand{\arraystretch}{1}
\begin{matrix}
(-1)^{\hat J\hat K}+(-1)^{\hat I\hat K+\hat I\hat J} 	& -(-1)^{\hat I+\hat J}((-1)^{\hat J\hat K}+(-1)^{\hat I\hat K+\hat I\hat J})\\
(-1)^{\hat J\hat K}+(-1)^{\hat I\hat K+\hat I\hat J} 	& -(-1)^{\hat I+\hat J}((-1)^{\hat J\hat K}+(-1)^{\hat I\hat K+\hat I\hat J})\\
(-1)^{\hat J\hat K}+(-1)^{\hat I\hat K+\hat I\hat J} 	& +(-1)^{\hat I+\hat J}((-1)^{\hat J\hat K}+(-1)^{\hat I\hat K+\hat I\hat J})\\
(-1)^{\hat J\hat K}+(-1)^{\hat I\hat K+\hat I\hat J} 	& +(-1)^{\hat I+\hat J}((-1)^{\hat J\hat K}+(-1)^{\hat I\hat K+\hat I\hat J})\\
\end{matrix}}\right. \\
\left.
{\renewcommand{\arraystretch}{1.3}\begin{matrix}
-2(-1)^{(\hat J+1)(\hat K+1)} 	& +2(-1)^{\hat I\hat J}(-1)^{(\hat I+1)(\hat K+1)} \\
+2(-1)^{(\hat J+1)(\hat K+1)} 	& -2(-1)^{\hat I\hat J}(-1)^{(\hat I+1)(\hat K+1)} \\
+2(-1)^{(\hat J+1)(\hat K+1)} 	& +2(-1)^{\hat I\hat J}(-1)^{(\hat I+1)(\hat K+1)} \\
-2(-1)^{(\hat J+1)(\hat K+1)} 	& -2(-1)^{\hat I\hat J}(-1)^{(\hat I+1)(\hat K+1)} \\
\end{matrix}}\right)
\end{multline}
We note that in the two special situations $c=0$ or $d=0$ the map $q_{c,d}$ is a self map if and only if $(-1)^{\hat I\hat J+\hat I\hat K+\hat J\hat K}=-1$. We then have 
$q_{c,0}(e_\mu)=(\sigma_{IK}\alpha^2+\sigma_{JK}\beta^2)\Gamma_\mu$, for example.

We solve the linear system by considering several cases by excluding the cases $c=0$ and $d=0$.

{\em Case 1a: }\ We consider the generic case $\hat K\hat I\hat J\neq0$ and $(K\cup I\cup J)^\complement\neq\emptyset$. 
We have to distinguish eight cases due to the value of $\hat K$, $\hat I$, and $\hat J$ mod 2. The matrices in these cases are\footnote{The rows and columns are as before. In addition, the subscript indicates the value of $(\hat K,\hat I,\hat J)\in\ZZ^3_2$.} 
\begin{equation}\label{alle}\begin{aligned}
& \left( \begin{smallmatrix} 1&-1& 1&-1\\ 1&-1&-1& 1\\ 1& 1&-1&-1\\ 1& 1& 1& 1\end{smallmatrix} \right)_{000},
&&\left( \begin{smallmatrix} 1& 1&-1&-1\\ 1& 1& 1& 1\\ 1&-1& 1&-1\\ 1&-1&-1& 1\end{smallmatrix}\right)_{001},
&&\left( \begin{smallmatrix} 1& 1& 1& 1\\ 1& 1&-1&-1\\ 1&-1&-1& 1\\ 1&-1& 1&-1\end{smallmatrix}\right)_{010},\\
& \left( \begin{smallmatrix} 0& 0&-1&-1\\ 0& 0& 1& 1\\ 0& 0& 1&-1\\ 0& 0&-1& 1\end{smallmatrix}\right)_{011},
&&\left( \begin{smallmatrix} 1&-1&-1& 1\\ 1&-1& 1&-1\\ 1& 1& 1& 1\\ 1& 1&-1&-1\end{smallmatrix}\right)_{100},
&&\left( \begin{smallmatrix} 0& 0&-1& 1\\ 0& 0& 1&-1\\ 0& 0& 1& 1\\ 0& 0&-1&-1\end{smallmatrix}\right)_{101},\\
& \left( \begin{smallmatrix} 0& 0&-1& 1\\ 0& 0& 1&-1\\ 0& 0& 1& 1\\ 0& 0&-1&-1\end{smallmatrix}\right)_{110},
&&\left( \begin{smallmatrix} 0& 0&-1&-1\\ 0& 0& 1& 1\\ 0& 0& 1&-1\\ 0& 0&-1& 1\end{smallmatrix}\right)_{111}.
&&
\end{aligned}\end{equation}
The first three as well as the fifth systems have the unique solution $\alpha\beta=\alpha'\beta'=\alpha\beta'=\alpha'\beta=0$ which yield $\alpha=\alpha'=0$ or $\beta=\beta'=0$. 
Therefore, the result is a monomial pair according to Proposition \ref{monomial}. 
Systems four and six to eight yield $\alpha\beta'=\alpha'\beta=0$ with solutions $\alpha=\alpha'=0$ or $\beta=\beta'=0$ and yield monomial pairs, too. 
In the generic Case 1a there are no new quadratic Clifford pairs beside the monomial ones. 

{\em Case 1b: }\ Turning to the case $\hat K=0$, $\hat I\hat J\neq0$, and $(I\cup J)^\complement\neq\emptyset$ we can take over the first four matrices of (\ref{alle}) with last row erased 
\begin{equation}\begin{aligned}
& \left( \begin{smallmatrix} 1&-1& 1&-1\\ 1&-1&-1& 1\\ 1& 1&-1&-1\end{smallmatrix}\right)_{000},
&&\left( \begin{smallmatrix} 1& 1&-1&-1\\ 1& 1& 1& 1\\ 1&-1& 1&-1\end{smallmatrix}\right)_{001},
&&\left( \begin{smallmatrix} 1& 1& 1& 1\\ 1& 1&-1&-1\\ 1&-1&-1& 1\end{smallmatrix}\right)_{010},\\
&\left(  \begin{smallmatrix} 0& 0&-1&-1\\ 0& 0& 1& 1\\ 0& 0& 1&-1\end{smallmatrix}\right)_{011}.
&&&&
\end{aligned}\end{equation}
The fourth system only yields monomial pairs due to its solution $\alpha\beta'=\beta\alpha'=0$. 
The first system yields $\alpha\beta=\alpha'\beta'=\alpha\beta'=\alpha'\beta$ which is $\alpha(\beta-\beta')=\beta(\alpha-\alpha')=\alpha'(\beta-\beta')=\beta'(\alpha-\alpha')=0$. 
In addition to the monomial ones we get a further pair connected by $\alpha=\alpha'$ and $\beta=\beta'$. Therefore, the new solution is of the form $(c,\bar c)$. 
The remaining two system have similar solutions and again yield the quadratic Clifford pair $(c,\bar c)$. 

The quadratic Clifford pairs in Case 1b are $\big(c,\bar c\big)$ 
with $c=\alpha\Gamma_I+\beta\Gamma_J$ and $\hat I\hat J\equiv 0\,{\rm mod}\,2$.

{\em Case 2a: }\ If we consider $(K\cup I\cup J)^\complement=\emptyset$ and $\hat I\hat J\neq0$ as well as $\hat K\neq0$ we have to look at all of (\ref{alle}) with third row erased:
\begin{equation}\begin{aligned}
&\left(  \begin{smallmatrix} 1&-1& 1&-1\\ 1&-1&-1& 1\\ 1& 1& 1& 1\end{smallmatrix}\right)_{000},
&&\left( \begin{smallmatrix} 1& 1&-1&-1\\ 1& 1& 1& 1\\ 1&-1&-1& 1\end{smallmatrix}\right)_{001},
&&\left( \begin{smallmatrix} 1& 1& 1& 1\\ 1& 1&-1&-1\\ 1&-1& 1&-1\end{smallmatrix}\right)_{010},\\
&\left(  \begin{smallmatrix} 0& 0&-1&-1\\ 0& 0& 1& 1\\ 0& 0&-1& 1\end{smallmatrix}\right)_{011},
&&\left( \begin{smallmatrix} 1&-1&-1& 1\\ 1&-1& 1&-1\\ 1& 1&-1&-1\end{smallmatrix}\right)_{100},
&&\left( \begin{smallmatrix} 0& 0&-1& 1\\ 0& 0& 1&-1\\ 0& 0&-1&-1\end{smallmatrix}\right)_{101},\\
&\left(  \begin{smallmatrix} 0& 0&-1& 1\\ 0& 0& 1&-1\\ 0& 0&-1&-1\end{smallmatrix}\right)_{110},
&&\left( \begin{smallmatrix} 0& 0&-1&-1\\ 0& 0& 1& 1\\ 0& 0&-1& 1\end{smallmatrix}\right)_{111}.
&&
\end{aligned}\label{nichtalle}\end{equation}
A discussion similar to the one before yields the following result. The only systems that admit additional non-monomial quadratic Clifford pairs are the first three and the fifth. The pairs in these cases are of the form $(c,-\bar c)$. 
Due to $(I\cup J\cup K)^\complement=\emptyset$ we have $\hat I+\hat J+\hat K=\dim(V)$. If $\dim(V)$ is odd we may assume $\Gamma_I\Gamma_J\Gamma_K=\mathbbm{1}$ or $\Gamma_I\Gamma_K=\Gamma_J$ and $\Gamma_J\Gamma_K=(-1)^{\hat K}\Gamma_J$ such that $e=a\Gamma_J+(-)^{\hat K}b\Gamma_I$ and the pair is of type 1b. If dim V is even we have $\Gamma_I\Gamma_J\Gamma_K\sim\Gamma^*$ such that we may assume $c=(\alpha\Gamma_I+\beta\Gamma^*)\Gamma_K$.

In Case 2a the quadratic Clifford pairs are $\big(c,-\bar c\big)$ 
with $c=(\alpha\Gamma_K+\beta\Gamma^*)\Gamma_I$ 
and $\hat K\equiv \hat I\equiv \hat J\equiv0\,{\rm mod}\,2$.

{\em Case 2b: }\ The case $(I\cup J)^\complement=\emptyset$, $\hat I\hat J\neq0$, and $\hat K=0$ (i.e.\ $\Gamma_J\sim \Gamma_I\Gamma^*$) can be read from the first four systems of (\ref{nichtalle}) with last row erased. Furthermore, we may exclude systems two and three because in this case $\Gamma^*\sim \mathbbm{1}$ and the Clifford pair is a priori monomial. 
\begin{equation}\begin{aligned}
&\left(  \begin{smallmatrix} 1&-1& 1&-1\\ 1&-1&-1& 1\end{smallmatrix}\right)_{000},
&&\left(  \begin{smallmatrix} 0& 0&-1&-1\\ 0& 0& 1& 1\end{smallmatrix}\right)_{011}.
\end{aligned}\end{equation}
For $\hat I\equiv\hat J\equiv 0\,{\rm mod}\,2$ the solution is $\alpha=\beta=\alpha'\beta'$ and $\alpha\beta'=\alpha'\beta$. Non-monomial pairs are related to solutions with $\alpha\beta\alpha'\beta'\neq0$. In this situation the coefficients obey $\alpha^2=\alpha'^2$ and $\beta^2=\beta'^2$ such that the pairs are given by $\big(c,\pm c\big)$ with $c=(\alpha+\beta\Gamma^*)\Gamma_I$.

The situation is less restrictive for $\hat I\equiv\hat J\equiv 1\,{\rm mod}\, 2$ with $\Gamma_J=\Gamma^*\Gamma_I$. 
The system has solution $\alpha\beta'=-\alpha'\beta$ which yields non-monomial pairs at most for $\alpha\beta\alpha'\beta'\neq0$. 
For  $\frac{\alpha}{\beta}=-\frac{\alpha'}{\beta'}=\cot(\phi)e^{i\psi}$ we get 
$c= \frac{\beta}{\sin(\phi)}(\cos(\phi)e^{i\psi}\Gamma_I+\sin(\phi)\Gamma_J)$ and 
$d=-\frac{\beta'}{\sin(\phi)}(\cos(\phi)e^{i\psi}\Gamma_I+\sin(\phi)\Gamma_J)$. 

The quadratic Clifford pairs in Case 2b are given by 
$\big(c,\pm c\big)$ 
with $c=(\alpha+\beta\Gamma^*)\Gamma_I$ 
for  $\hat I\equiv 0\,{\rm mod}\,2$ 
and  $\big( \alpha(\cos(\phi)e^{i\psi} +\sin(\phi)\Gamma^*)\Gamma_I, \alpha'(\cos(\phi)e^{i\psi}-\sin(\phi)\Gamma^*)\Gamma_I\big)$ 
for $\hat I\equiv 1\,{\rm mod}\,2$.

{\em Case 3a: }\ We turn to the cases with $\hat I=0$ and consider first $\hat K\neq0$ and $(J\cup K)^\complement\neq\emptyset$. In this case we are left with systems one, two, five and six of (\ref{alle}) with second row erased:
\begin{equation}\label{117}\begin{aligned}
& \left( \begin{smallmatrix} 1&-1& 1&-1\\ 1& 1&-1&-1\\ 1& 1& 1& 1\end{smallmatrix}\right)_{000},
&&\left( \begin{smallmatrix} 1& 1&-1&-1\\ 1&-1& 1&-1\\ 1&-1&-1& 1\end{smallmatrix}\right)_{001},
&&\left( \begin{smallmatrix} 1&-1&-1& 1\\ 1& 1& 1& 1\\ 1& 1&-1&-1\end{smallmatrix}\right)_{100},\\
& \left( \begin{smallmatrix} 0& 0&-1& 1\\ 0& 0& 1& 1\\ 0& 0&-1&-1\end{smallmatrix}\right)_{101}
&&
&&
\end{aligned}\end{equation}
There is no non-monomial solution of the fourth system. The first system has solution $\alpha\beta=-\alpha'\beta'=-\alpha\beta'=\beta\alpha'$ with non-monomial pair $\big( (\alpha+\beta\Gamma_J)\Gamma^K,(\alpha-\beta\Gamma_J)\Gamma_K\big)$. In the same way second system yields $\alpha\beta=\alpha'\beta'=\alpha\beta'=\beta\alpha'$ with $\big(\alpha+\beta\Gamma_J)\Gamma^K,(\alpha+\beta\Gamma_J)\Gamma_K\big)$, and the third one yields $\alpha\beta=-\alpha'\beta'=\alpha\beta'=-\alpha'\beta$ with  $\big( (\alpha+\beta\Gamma_J)\Gamma_K,(-\alpha+\beta\Gamma_J)\Gamma_K\big)$.

In Case 3a the relevant pairs are $(c_+,\bar c_-)$ 
with $c_\pm=(\alpha\pm \beta\Gamma_J)\Gamma_K$
and  $\hat K\hat J\equiv 0\,{\rm mod}\,2$.

{\em Case 3b: }\ Next we turn to the case $\hat I=\hat K=0$ and $J^\complement\neq\emptyset$. In this case we may consider the first two systems from Case 3a with last row erased, i.e.\ 
\begin{equation}\begin{aligned}
& \left( \begin{smallmatrix} 1&-1& 1&-1\\ 1& 1&-1&-1\end{smallmatrix}\right)_{000},
&&\left( \begin{smallmatrix} 1& 1&-1&-1\\ 1&-1& 1&-1\end{smallmatrix}\right)_{001},
\end{aligned}\end{equation}
The solutions of both systems are $\beta'(\alpha-\alpha')=0$ and $\beta(\alpha-\alpha')=0$ with non-monomial pairs defined by the solution
$\alpha=\alpha'$. In this case the pair is given by $(\alpha+\beta\Gamma_J,\alpha+\beta'\Gamma_J)$. The case $\hat J=0$ is the same. 

The quadratic Clifford pair in Case 3b is $(\alpha+\beta\Gamma_J,\alpha+\beta'\Gamma_J)$. 

{\em Case 4a: }\ This case is $\hat I=0$, $\hat K\neq0$ and $(J\cup K)^\complement=\emptyset$. The systems in this case are 
\begin{equation}\begin{aligned}
& \left( \begin{smallmatrix} 1&-1& 1&-1\\ 1& 1& 1& 1\end{smallmatrix}\right)_{000},
&&\left( \begin{smallmatrix} 1& 1&-1&-1\\ 1&-1&-1& 1\end{smallmatrix}\right)_{001},
&&\left( \begin{smallmatrix} 1&-1&-1& 1\\ 1& 1&-1&-1\end{smallmatrix}\right)_{100},\\
& \left( \begin{smallmatrix} 0& 0&-1& 1\\ 0& 0&-1&-1\end{smallmatrix}\right)_{101}
&&
&&
\end{aligned}\end{equation}

There are non-monomial solutions for the systems one, two, and three. The resulting pairs are $\big((\alpha+\beta\Gamma_J)\Gamma_K,(\alpha'-(-1)^{\dim(V)}\beta\Gamma_J)\Gamma_K\big)$. We may consider $\Gamma_J=\Gamma^*\Gamma_K$ such that we may rewrite the pair as 
$\big(\alpha\Gamma_K+\beta\Gamma^*,\alpha'\Gamma_K-(-1)^{\dim(V)}\beta\Gamma^*\big)$. For $\dim(V)$ odd we are left with Case 3b. 

Case 4a yields the quadratic Clifford pairs $\big(\alpha\Gamma_K+\beta\Gamma^*,\alpha'\Gamma_K-\beta\Gamma^*\big)$.

{\em Case 4b: }\ The last case is $\hat I=\hat K=0$ and $J^\complement=\emptyset$, i.e.\ $\Gamma_J=\Gamma^*$. We only consider the case $\dim(V)=\hat J$ even, because in the other case $c$ and $d$ are scalars. So the two elements are of the form $c=\alpha+\beta\Gamma^*$, $d=\alpha'+\beta'\Gamma^*$. The obstruction for $(c,d)$ being a quadratic Clifford pair is  $\alpha\beta-\alpha'\beta'+\alpha\beta'-\beta\alpha'=(\alpha-\alpha')(\beta+\beta')=0$. 

In Case 4b the non-monomial quadratic Clifford pairs are of the form
$(\alpha+\beta\Gamma^*,\alpha+\beta'\Gamma^*)$ or $(\alpha+\beta\Gamma^*,\alpha'-\beta\Gamma^*)$.

A generalization of the results from Propositions \ref{monomial} and \ref{moregen1} is given by the following approach. Let $\left(\bigcup_\alpha I_\alpha\right)^\complement=\emptyset$ and $c=\sum_\alpha a_\alpha\Gamma_{I_\alpha}$ as well as $d=\sum_\alpha \ep_\alpha a_\alpha\Gamma_{I_\alpha}$ with $\ep_\alpha\in\{\pm1\}$. Then 
\begin{align*}
c^2\Gamma_\mu 
=\ & \Big(\sum_\alpha a_\alpha\Gamma_{I_\alpha}\Big)^2\Gamma_\mu 
=  \sum_\alpha a_\alpha^2  +\sum_{\alpha<\beta} a_\alpha a_\beta \big(1+(-1)^{\hat\alpha\hat\beta}\big)\Gamma_{I_\alpha I_\beta}\Gamma_\mu
\end{align*}
and
\begin{align*}
\Gamma_\mu d^2
=\ &\Gamma_\mu\Big(\sum_\alpha \ep_\alpha a_\alpha\Gamma_{I_\alpha}\Big)^2 \\
=\	& \sum_\alpha a_\alpha^2 \Gamma_\mu  +\sum_{\hat \alpha<\hat \beta} \ep_\alpha\ep_\beta \theta^\mu_\alpha\theta^\mu_\beta a_\alpha a_\beta  (-1)^{\hat\alpha+\hat\beta}\big(1+(-1)^{\hat\alpha\hat\beta}\big)\Gamma_{I_\alpha I_\beta}\Gamma_\mu
\end{align*}
as well as
\begin{align*}
&\Big(\sum_\alpha a_\alpha\Gamma_{I_\alpha}\Big)\Gamma_\mu \Big(\sum_\alpha\ep_\alpha a_\alpha\Gamma_{I_\alpha}\Big) \\
=\	& \sum_\alpha \ep_\alpha\theta^\mu_\alpha(-1)^{\hat\alpha} a_\alpha^2 \Gamma_\mu
	+\sum_{\hat\alpha<\hat \beta} a_\alpha a_\beta \big(\ep_\beta\theta^\mu_\beta(-1)^{\hat\beta}+\ep_\alpha\theta^\mu_\alpha(-1)^{\hat \alpha + \hat\alpha\hat\beta}\big)
	\Gamma_{I_\alpha I_\beta}\Gamma_\mu
\end{align*}
with $\theta^\mu_\alpha:=\theta^\mu_{I_\alpha}$ and $\hat \alpha:=\hat I_\alpha$.
Due to Remark \ref{a=0} we may assume that $\hat\alpha>0$. Therefore, our approach yields 
\begin{equation}\label{f}
q_{c,d}(e_\mu)= 2\sum_\alpha \big(1-\ep_\alpha\theta^\mu_\alpha(-1)^{\hat\alpha}\big) a_\alpha^2 \Gamma_\mu
		+ (\ldots)\,.
\end{equation}
The further summands collected in "$(\ldots)$" in (\ref{f}) are of order higher than one. Therefore, $(c,d)$ is a quadratic Clifford pair if and only if the $\Gamma_{I_\alpha I_\beta}\Gamma_\mu$-component of $q_{c,d}(e_\mu)$ vanishes. The coefficient of this contribution is given by 
\begin{equation}\label{val}
\big(1+(-1)^{\hat\alpha\hat\beta}\big)
	\big(1 +\ep_\alpha\ep_\beta\theta^\mu_\alpha\theta^\mu_\beta(-1)^{\hat\alpha+\hat\beta}\big)
-2\big(\ep_\beta\theta^\mu_\beta(-1)^{\hat\beta}
		+\ep_\alpha\theta^\mu_\alpha(-1)^{\hat\alpha+\hat\alpha\hat\beta}\big)	\,.
\end{equation}
In the following Table \ref{table:app} we list this value all possible combinations $\theta_\alpha^\mu$, $\hat\alpha$, $\hat\beta$.

\noindent\makebox[\textwidth]{%
\begin{minipage}{\textwidth}\captionof{table}{(\ref{val}) for given $\theta_\alpha^\mu$, $\hat\alpha$, $\hat\beta$}\label{table:app}
\vspace*{-2ex}{$\displaystyle 
\renewcommand{\arraystretch}{1.2}
\begin{array}{ll@{$\,:\qquad$}l}
\multicolumn{3}{l}{\text{1)}\ \hat\alpha\equiv \hat\beta\equiv 0\,{\rm mod}\,2} \\
\quad\text{1a)} & \theta^\mu_\alpha=1,\theta^\mu_\beta=1  
	& 2 \big(1 +\ep_\alpha\ep_\beta-\ep_\beta-\ep_\alpha\big) =2(1-\ep_\beta)(1-\ep_\alpha)\\
\quad\text{1b)} & \theta^\mu_\alpha=-1,\theta^\mu_\beta=1 
	& 2 \big(1 -\ep_\alpha\ep_\beta-\ep_\beta +\ep_\alpha\big)=2(1-\ep_\beta)(1+\ep_\alpha)\\
\quad\text{1c)} & \theta^\mu_\alpha=1,\theta^\mu_\beta=-1 
	& 2 \big(1 -\ep_\alpha\ep_\beta+\ep_\beta -\ep_\alpha\big)=2(1+\ep_\beta)(1-\ep_\alpha)\\[1ex]
\multicolumn{3}{l}{\text{2)}\ \hat\alpha\equiv 1\,{\rm mod}\,2,\hat\beta\equiv 0\,{\rm mod}\,2} \\
\quad\text{2a)} & \theta^\mu_\alpha=1,\theta^\mu_\beta=1  
	& 2\big(1 -\ep_\alpha\ep_\beta-\ep_\beta+\ep_\alpha\big)=2(1-\ep_\beta)(1+\ep_\alpha)	\\
\quad\text{2b)} & \theta^\mu_\alpha=-1,\theta^\mu_\beta=1 
	& 2\big(1 +\ep_\alpha\ep_\beta-\ep_\beta-\ep_\alpha\big)=2(1-\ep_\beta)(1-\ep_\alpha)	\\
\quad\text{2c)} & \theta^\mu_\alpha=1,\theta^\mu_\beta=-1 
	& 2\big(1 +\ep_\alpha\ep_\beta+\ep_\beta+\ep_\alpha\big)=2(1+\ep_\beta)(1+\ep_\alpha)	\\[1ex]
\multicolumn{3}{l}{\text{3)}\ \hat\alpha\equiv \hat\beta\equiv 1\,{\rm mod}\,2} \\
\quad\text{3a)} & \theta^\mu_\alpha=1,\theta^\mu_\beta=1  
	&  2(\ep_\beta-\ep_\alpha)	\\
\quad\text{3b)} & \theta^\mu_\alpha=-1,\theta^\mu_\beta=1 
	& 2(\ep_\beta+\ep_\alpha)	\\
\quad\text{3c)} & \theta^\mu_\alpha=1,\theta^\mu_\beta=-1 
	& -2(\ep_\beta+\ep_\alpha)		
%
\end{array}
$}
\end{minipage}
}

The proof of Proposition \ref{hom} is now done by the following considerations. 

Firstly, suppose that the amount of summands in $c$ is bigger or equal to four, such that all sub cases are present if one is present at all. We see that 1) and 2) are fulfilled for $\ep_\alpha=(-1)^{\hat \alpha}$. 
Moreover, 3) can not be fulfilled if it is present such that the amount of odd summands is at most one.  
If the amount of summands is three a similar argument excludes the case of two odd summands: 2b) and 2c) would yield $\ep_{\text{even}}=-\ep_{\text{odd}_1}=-\ep_{\text{odd}_2}$ and 3) $\ep_{\text{odd}_1}=-\ep_{\text{odd}_2}$.

Therefore, in even dimensions our approach yields a quadratic Clifford pair if and only if $\hat I_\alpha$ is even for all $\alpha$, i.e.\ $c=d$. 
In odd dimensions all $\hat I_\alpha$ would be even but one. Nevertheless, in both situations we have $\ep_\alpha=(-1)^{\hat \alpha}$

In the case of two summands the subcases denoted by a) in the above list do not appear and we recover the pseudo-monomial solutions of Table \ref{table2}. In either case we have
\begin{equation}
q_{c,d}(e_\mu)=2\sum_\alpha \big(1-\theta^\mu_\alpha\big) a_\alpha^2\Gamma_\mu\,.
\end{equation}

As promised after Proposition \ref{hom-p} we will list next the conditions such that the pair $(c,d)$ with 
\begin{equation*}
\begin{aligned}
c&=\sum_\alpha c_\alpha\Gamma_{I_\alpha}+\Gamma^*\sum_\alpha\hat c_\alpha\Gamma_{I_\alpha},\quad 
d=\sum_\alpha \ep_\alpha c_\alpha\Gamma_{I_\alpha}+\Gamma^*\sum_\alpha\hat\ep_\alpha\hat c_\alpha\Gamma_{I_\alpha}
\end{aligned}
\end{equation*}
yields a quadratic Clifford pair:
\begin{align*}
0=\ & c_\alpha \hat c_\beta \Big(
			\big(1-(-1)^{\hat \alpha}\theta^\mu_\alpha\ep_\alpha\big)
			\big((-1)^{\hat\alpha\hat\beta}+(-1)^{\hat\alpha+\hat\beta}\theta^\mu_\beta\hat\ep_\beta\big)\\
&\qquad	 +(-1)^{\hat\alpha}(-1)^{\hat\alpha\hat\beta}
			\big(1+(-1)^{\hat\beta}\theta^\mu_\beta\hat\ep_\beta\big)
			\big((-1)^{\hat\alpha\hat\beta}-\theta^\mu_\alpha\ep_\alpha\big)
			\Big) \\
  & +\hat c_\alpha c_\beta \Big( 
   			(-1)^{\hat\alpha\hat\beta}
   			\big(1-(-1)^{\hat\beta}\theta^\mu_\beta\ep_\beta\big)
   			\big((-1)^{\hat\alpha\hat\beta}+(-1)^{\hat\alpha+\hat\beta}\theta^\mu_\alpha\hat\ep_\alpha\big)\\
&\qquad	 +(-1)^{\hat\beta}
   			\big(1+(-1)^{\hat\alpha}\theta^\mu_\alpha\hat\ep_\alpha\big)
   			\big((-1)^{\hat\alpha\hat\beta}-\theta^\mu_\beta\ep_\beta\big)
   			\Big)\,,\\
0=\ & c_\alpha  c_\beta \Big(
			\big(1-(-1)^{\hat\beta}\theta^\mu_\beta\ep_\beta\big)
			\big(1-(-1)^{\hat\alpha}(-1)^{\hat\alpha\hat\beta}\theta^\mu_\alpha\ep_\alpha\big)\\
&\qquad	 +(-1)^{\hat\alpha\hat\beta}
			\big(1-(-1)^{\hat\alpha}\theta^\mu_\alpha\ep_\alpha\big)
			\big(1-(-1)^{\hat\beta}(-1)^{\hat\alpha\hat\beta}\theta^\mu_\beta\ep_\beta\big)
			\Big) \\
  & +\hat c_\alpha \hat c_\beta \Big( 
   			(-1)^{\hat\alpha}
   			\big(1+(-1)^{\hat\beta}\theta^\mu_\beta\hat\ep_\beta\big)
   			\big(1+(-1)^{\hat\beta}(-1)^{\hat\alpha\hat\beta}\theta^\mu_\alpha\hat\ep_\alpha\big)\\
&\qquad	 +(-1)^{\hat\beta}(-1)^{\hat\alpha\hat\beta}
   			\big(1+(-1)^{\hat\alpha}\theta^\mu_\alpha\hat\ep_\alpha\big)
   			\big(1+(-1)^{\hat\alpha}(-1)^{\hat\alpha\hat\beta}\theta^\mu_\beta\hat\ep_\beta\big)
   			\Big)\,,\\
0=\ &\sum_\alpha c_\alpha\hat c_\alpha \Big(
		    \big(1-(-1)^{\hat\alpha}\theta^\mu_\alpha\ep_\alpha\big)
			\big(1+\theta^\mu_\alpha\hat \ep_\alpha\big)
 +(-1)^{\hat\alpha}
			\big(1+(-1)^{\hat\alpha}\theta^\mu_\alpha\hat\ep_\alpha\big)
			\big(1-\theta^\mu_\alpha\ep_\alpha\big)
			\Big)\,.
\end{align*}
A discussion similar to the one before yields the strong restrictions described in Proposition \ref{hom-p} if we consider $\ep_\alpha=-\hat\ep_\alpha=(-1)^{\hat\alpha}$. 

\subsection{Proof of (\ref{19})-(\ref{22})}\label{app3}

We consider the following matrix form of the images of the Clifford map $\rho$:
\begin{equation}
\begin{gathered}
\rho(v^*)=\frac{1}{\sqrt{2}}{\renewcommand{\arraystretch}{1.2}
\begin{pmatrix}0&Bv\\0&0\end{pmatrix}}\,,\quad 
\rho(w)={\renewcommand{\arraystretch}{1.3}\begin{pmatrix}\bar \rho_{11}(v) & \rho_{12}(v)\\\bar \rho_{21}(w) &\rho_{22}(w)\end{pmatrix}}\,,\\
\rho(e_+)={\renewcommand{\arraystretch}{1.3}\begin{pmatrix}\bar a_{11}&\sqrt{2}a_{12}\\\sqrt{2}\bar a_{21}&a_{22}\end{pmatrix}}\,,\quad
\rho(e_-)={\renewcommand{\arraystretch}{1.3}\begin{pmatrix}\bar b_{11}&\sqrt{2}b_{12}\\\sqrt{2}\bar b_{21}&b_{22}\end{pmatrix}}\,.
\end{gathered}
\end{equation}
The $V^*$-equivariance is expressed as
\begin{equation}\label{a1}
\begin{aligned}
{}[\rho(v^*),\rho(e_+)]
=\ &  	\frac{1}{\sqrt{2}}{\renewcommand{\arraystretch}{1.3}\begin{pmatrix}\sqrt{2}Bv a_{21} & Bv a_{22}-a_{11}Bv\\0 &-\sqrt{2}a_{21}Bv \end{pmatrix}} \\
=\ &		0\,,
\end{aligned}
\end{equation}
\begin{equation}\label{a2}
\begin{aligned}
{}[\rho(v^*),\rho(w)]
=\ &\frac{1}{\sqrt{2}}{\renewcommand{\arraystretch}{1.3}\begin{pmatrix} Bv\bar \rho_{21}(w) & Bv\rho_{22}(w)-\bar \rho_{11}(w) Bv\\0 &-\bar \rho_{21}(w)Bv \end{pmatrix}}\\
=\ &-\langle Bv,w\rangle {\renewcommand{\arraystretch}{1.3}\begin{pmatrix} \bar a_{11} &\sqrt{2}a_{12}\\ \sqrt{2}\bar a_{21} &a_{22}\end{pmatrix}}\,,
\end{aligned}
\end{equation}
\begin{equation}\label{a3}
\begin{aligned}
{}[\rho(v^*),\rho(e_-)]
=\ & \frac{1}{\sqrt{2}}{\renewcommand{\arraystretch}{1.3}\begin{pmatrix}\sqrt{2}Bv b_{21}  & Bv b_{22}-\bar b_{11}Bv\\0 &-\sqrt{2}\bar b_{21}Bv \end{pmatrix}} \\
=\ & {\renewcommand{\arraystretch}{1.3}\begin{pmatrix}\bar \rho_{11}(Bv) & \rho_{12}(Bv)\\\bar \rho_{21}(Bw) &\rho_{22}(Bw)\end{pmatrix}}\,.
\end{aligned}
\end{equation}
(\ref{a1}) yields {$\bar a_{21}=0$} and (\ref{a3}) yields {$\bar\rho_{21}=0$}. Then (\ref{a2}) immediately yields {$\bar a_{11}=a_{22}=0$}. 

Furthermore, (\ref{a3}) yields {$\rho_{12}(v)= -\frac{1}{\sqrt{2}}s_{\bar b_{11},b_{22}}(v)$} as well as {$\bar \rho_{11}(v)=v\bar b_{21}$} and {$\rho_{22}(v)=-\bar b_{21}v$}. The last two equations together with (\ref{a2}) yield $ v\bar b_{21}w+w\bar b_{21}v=2\langle v,w \rangle a_{12}$. 
We consider a basis $\Gamma_\mu$ for which the latter is written as {$\Gamma_{\{\mu}\bar b_{21}\Gamma_{\nu\}} = g_{\mu\nu}a_{12}$} and get after taking the trace $a_{12}= \frac{1}{n}\sum_\mu\Gamma_\mu  \bar b_{21} \Gamma^\mu$.

\end{appendix}


\begin{thebibliography}{99}

\begin{small} 
\bibitem{Cortes}
Dimitry V.~Alekseevsky, Vicente Cort\'es, Chandrashekar Devchand, and Uwe Semmelmann: 
\newblock Killing Spinors are Killing vector fields in Riemannian supergeometry.
\newblock  {\em J.\ Geom.\ Phys.} {\bf 26} (1998), no.~1-2, 37–50.

\bibitem{BaumKath99}
Helga Baum  and Ines Kath:
\newblock Parallel spinors and holonomy groups on pseudo-Riemannian spin manifolds.
\newblock {\em Ann.\ Global Anal.\ Geom.} {\bf 17} (1999), no.~1, 1–17. 

\bibitem{Baum12}
Helga Baum, Kordian L\"arz, and Thomas Leistner:
\newblock  On the full holonomy group of manifolds with special holonomy.
\newblock arXiv:1204.5657 [math.DG], 2012.

\bibitem{Besse}
Arthur L.\ Besse:
\newblock{\em Einstein Manifolds.} (Classics in Mathematics)
\newblock Springer Verlag, Reprint of the 1987 ed., 2008.

\bibitem{CW70}
Michel Cahen and Nolan Wallach:
\newblock Lorentzian symmetric spaces.
\newblock {\em Bull.\ Amer.\ Math.\ Soc.} {\bf 76} (1970), 585-591.

\bibitem{CS}
Andreas \v{C}ap and Jan Slov\'{a}k:
\newblock {\em Parabolic Geometries I}. 
\newblock Mathematical Surveys and Monographs, 154. American Mathematical Society, 2009.

\bibitem{Chevalley}
Claude Chevalley:
\newblock {\em The Algebraic Theory of Spinors and Clifford Algebras} (Collected Works, Vol.~2).
\newblock Springer-Verlag, Berlin, 1997.

\bibitem{CheKo84}
Piotr T.~Chru\'{s}ciel and Jerzy Kowalski-Glikman:
\newblock The isometry group and {K}illing spinors for the pp wave space-time in $D=11$ supergravity.
\newblock {\em Phys.\ Lett.\ B} {\bf 149} (1984), no.~1-3, 107-110.

\bibitem{Fig1}
Jos{\'e} Figueroa-O'{}Farrill:
\newblock Lorentzian symmetric spaces in supergravity.
arXiv:math/0702205v1 [math.DG], 2007.

\bibitem{MeFig04}
Jos{\'e} Figueroa-O'{}Farrill, Patrick Meessen, and Simon Philip:
\newblock Supersymmetry and homogeneity of M-theory backgrounds.
\newblock {\em Class.\ Quant.\ Grav.} {\bf 22} (2005), 207-226.

\bibitem{FigPapado1}
Jos{\'e} Figueroa-O'{}Farrill and George Papadopoulos:
\newblock Homogeneous fluxes, branes and a maximally supersymmetric solution of {M}-theory.
\newblock {\em J.~High Energy Phys.} {\bf 8} (2001) Paper 36:26.

\bibitem{Hustler}
Jos\'{e} Figueroa-O'Farrill and  Noel Hustler:
\newblock Symmetric backgrounds of type IIB supergravity.
\newblock  {\em Classical Quantum Gravity} {\bf 30} (2013), no.~4, 045008, 36 pp.

\bibitem{Harvey}
F.~Reese Harvey:
\newblock {\em Spinors and Calibrations.} (Perspectives in Mathematics)
\newblock Academic Press Inc., 1990. 
 
\bibitem{KathOlb}
Ines Kath and Martin Olbrich:
\newblock On the structure of pseudo-Riemannian symmetric spaces.
\newblock {\em Transform.\ Groups} {\bf 14} (2009), no.~4, 847-885. 

\bibitem{KlinkerSSKS}
Frank Klinker:
\newblock Supersymmetric Killing structures.
\newblock {\em Comm.\ Math.\ Phys.} {\bf 255} (2005), no.~2, 419-467.

\bibitem{KlinkerTor}
Frank Klinker:
\newblock The torsion of spinor connections and related structures. 
\newblock {\em SIGMA Symmetry Integrability Geom.\ Methods Appl.} {\bf 2} (2006), Paper 077, 28 pp.

\bibitem{KlinkerDeform}
Frank Klinker: 
\newblock SUSY structures on deformed supermanifolds.
\newblock {\em  Differential Geom. Appl.} {\bf 26} (2008), no.~5, 566-582.

\bibitem{KN2}
Shoshichi Kobayashi and Katsumi Nomizu.:
\newblock {\em Foundations of Differential Geometry, Vol.~II.}
\newblock Wiley Classics Library, John Wiley \& Sons, Inc, 1996.

\bibitem{LawMich}
H. Blaine Lawson, Jr.\ and Marie-Louise Michelsohn:
\newblock {\em Spin geometry.} (Princeton Mathematical Series, 38)
\newblock Princeton University Press, 1989.

\bibitem{Leist07}
Thomas Leistner: 
\newblock On the classification of Lorentzian holonomy groups.
\newblock {\em J.\ Differential Geom.} {\bf 76} (2007), no.~3, 423-484.

\bibitem{Lou}
Pertti Lounesto:
\newblock {\em Clifford algebras and spinors.} (London Mathematical Society Lecture Note Series, 286) 
\newblock Cambridge University Press, 2.\ ed., 2001.

\bibitem{Meessen2002}
Patrick Meessen:
\newblock Small note on pp-wave vacua in 6 and 5 dimensions.
\newblock {\em Phys.\ Rev.\ D} {\bf 65} (2002), 087501.

\bibitem{Neukirchner}
Thomas Neukirchner: 
\newblock Solvable Pseudo-Riemannian Symmetric Spaces.
\newblock arXiv:math/0301326 [math.DG], 2003.

\bibitem{Putt}
Joseph Putter:
\newblock Maximal sets of anti-commuting skew-symmetric matrices.
\newblock {\em J.\ London Math.\ Soc.} {\bf 42} (1967), 303-308.  

\bibitem{santi1}
Andrea Santi:
\newblock Superizations of Cahen-Wallach symmetric spaces and spin representations of the Heisenberg algebra.
\newblock {\em J.\ Geom.\ Phys.} {\bf 60} (2010), 295-325.

\bibitem{Wu67}
Hung-Hsi Wu:
\newblock Holonomy groups of indefinite metrics.
\newblock {\em Pacific J.\ Math.} {\bf 20} (1967), no.~2, 351-392.

\end{small}
\end{thebibliography}

\end{document}